\documentclass[reqno]{amsart}

\usepackage{epsfig,amsmath, amsfonts,amsthm,amssymb}
\usepackage{dsfont}

\usepackage[english]{babel}

\usepackage[left=1.75cm, top=2.7cm, bottom=2.50cm,right=1.75cm]{geometry}

\usepackage{stackrel}
\usepackage{verbatim}
\usepackage{setspace}
\usepackage{mathrsfs}
\usepackage{bm}
\usepackage{color}
\usepackage{epsfig}
\usepackage{epstopdf}
\usepackage{graphicx}
\usepackage{graphics}

\usepackage{url}

\usepackage{bm}

\usepackage{colortbl}

 \newtheorem{thm}{Theorem}[section]
 
 \newtheorem{lem}[thm]{Lemma}
 
 \newtheorem{prop}[thm]{Proposition}
 \theoremstyle{definition}

 \newtheorem{rem}[thm]{Remark}
 \numberwithin{equation}{section}
\DeclareMathOperator*{\esssup}{ess\,sup}
\DeclareMathOperator*{\essinf}{ess\,inf}

\newcommand{\be}{\begin{equation}}
\newcommand{\ee}{\end{equation}}
\newcommand{\bq}{\begin{eqnarray}}
\newcommand{\eq}{\end{eqnarray}}

\newcommand{\half}{\frac{1}{2}}

    \def\ed{{\,\stackrel{\frak {D}}{=}\,}}

    \def\bbd{{\mathbb D}}

    \def\calF{{\mathcal F}}

    \def\calN{{\mathcal N}}

    \def\calT{{\mathcal T}}
    \def\bbr{{\mathbb R}}
    \def\bbe{{\mathbb E}}
    \def\bbp{{\mathbb P}}
    
    \def\bbn{{\mathbb N}}
    \def\bbz{{\mathbb Z}}

    \def\ed{\;{\stackrel{\frak {D}}{=}}\;}

    \def\d{{d}}

    \def\g{{g}}
    \def\h{{h}}
    
    \def\k{{k}}

    \def\s{{s}}

    \def\L{{L}}
    
    \def\N{{N}}

    \def\S{{S}}
    \def\T{{T}}

  \definecolor{Red}{rgb}{1.00, 0.00, 0.00}
    
    \definecolor{DRed}{rgb}{0.7, 0.3, 0.00}
    
    \definecolor{Green}{rgb}{0.2, 0.5, 0.2}
    
    \definecolor{Blue}{rgb}{0.00, 0.00, 1.00}
    
    \definecolor{PaleGrey}{rgb}{.6, .6, .6}

\title{Change-point detection for L\'evy processes}

\author{{Jos\'e E. Figueroa-L\'opez}}
\address{
Department of Mathematics\\
Washington University in St. Louis, St. Louis, MO 63130, USA}
\email{{\tt figueroa@math.wustl.edu}}
\thanks{The first author's research was supported in part by the NSF grants DMS-1561141 and DMS-1613016.}
\author{Sveinn \'Olafsson}
\address{Department of Statistics and Applied Probability, University of California, Santa Barbara, CA 93106, USA}
\email{{\tt olafsson@pstat.ucsb.edu}}

\begin{document}
\maketitle

\begin{abstract}
	
	
	Since the work of Page in the 1950s, the problem of detecting an abrupt change in the distribution of stochastic processes has received a great deal of attention. In particular, a deep connection has been established between Lorden's minimax approach to change-point detection and the widely used CUSUM procedure, first for discrete-time processes, and subsequently for some of their continuous-time counterparts. However, results for processes with jumps are still scarce, while the practical importance of such processes has escalated since the turn of the century. 
	In this work we consider the problem of detecting a change in the distribution of continuous-time processes with independent and stationary increments, i.e.\ L\'evy processes, and our main result shows that CUSUM is indeed optimal in Lorden's sense. 
	This is the most natural continuous-time analogue of the seminal work of Moustakides \cite{Moustakides} for 
	sequentially observed random variables that are assumed to be i.i.d.\ 
	before and after the change-point. 
	From a practical perspective, 
	the approach we adopt is appealing as it consists in approximating the continuous-time problem by a suitable sequence of change-point problems with equispaced sampling points, and for which a CUSUM procedure is shown to be optimal.
	
	
\vspace{0.2 cm}

\vspace{0.1 cm}

\end{abstract}

\section{Introduction}

Quickest detection is the problem of detecting, with as little delay as possible, a change in the probability distribution of a sequence of random measurements, and it has a wide range of applications in various branches of science and engineering, such as signal processing, supply chain management, {cybersecurity}, and finance (see \cite{Olympia} and references therein). The main result of this paper is an extension of a well known discrete-time quickest detection result of Moustakides \cite{Moustakides}, to an important class of continuous-time stochastic processes with jumps: L\'evy processes.

In the discrete-time setting, the change-point problem involves a sequence $(X_n)_{n\geq 1}$ of random observations whose statistical properties change at some unknown point in time $\tau$. In the simplest case, the observations $X_1,X_2,\dots,X_{\tau-1}$ are assumed to be independently drawn from one distribution, while $X_{\tau},X_{\tau+1},\dots$ are independently drawn from a different distribution.
The objective is then to detect the change-point $\tau$ as soon as possible, and the set of feasible detection strategies corresponds to the set of (extended real-valued) stopping times with respect to the observed sequence, with the understanding that a stopping time $T$ decides that the change-point $\tau$ has occurred at time $k$ when $T=k$. Naturally, the frequency of false alarms needs to be taken into account, so the design of detection procedures typically involves optimizing a trade-off between two types of performance indices, one quantifying the delay between the time a change occurs and {the time} it is detected, i.e., 
the random variable {$(T-\tau+1)^+$},
and the other being a measure of the frequency of false alarms, i.e., events of the type $\{T<\tau\}$. 

There are two main formulations of this optimization problem. 
The first of these is a Bayesian formulation in which the change-point is endowed with a prior distribution, usually a geometric distribution in discrete-time models or an exponential distribution in continuous-time models. This framework was first proposed with {a} linear delay penalty by Kolmogorov and Shiryayev \cite{Shir}, where the expected delay $\bbe(T-\tau+1)^+$ was to be minimized subject to an upper bound on the probability of a false alarm, $\bbp(T<\tau)$. 
{In applications there is typically limited information about the distribution of the change-point, and the second formulation is a more conservative minimax approach,}
first proposed in the linear delay penalty case by Lorden \cite{Lorden}, in which the change-point is considered to be deterministic and unknown.
In this formulation, the delay penalty is a worst-case measure of delay, taken over all possible realizations of the observations leading up to the change-point, and over all possible values of the change-point (see Eq.\ (\ref{Lorden}) for details), and false alarms are constrained {by a} lower bound on the mean time between such events. 


In this work we are concerned with the latter formulation,
which, whenever it can be optimized, tends to give rise to the CUSUM (cumulative sum) stopping rule, first proposed by Page \cite{Page} as a continuous inspection scheme in the 1950s. CUSUM is one of the most widely used detection schemes in practice, and is based on the first time the accumulated likelihood (or log-likelihood) breaches a certain barrier (see Eqs.\ (\ref{cs1})-(\ref{cs2})).
For a sequence of independent observations as described above, the asymptotic optimality of CUSUM, as the mean time between false alarms tends to infinity, was shown by Lorden \cite{Lorden} in 1971, and fifteen years later, Moustakides \cite{Moustakides} proved its optimality for any finite bound on the false alarm rate. Similar procedures were subsequently applied in \cite{HVPoor} with Lorden's linear criterion replaced by exponentially penalized detection delays. 

For continuous-time processes, the optimality of the CUSUM procedure for detecting a change in the drift of a Brownian motion was shown independently by several authors (see \cite{Beibel}, \cite{Moustakides2}, and \cite{Shir2}). More generally, its optimality for detecting a change in the drift of It\^o processes was shown in \cite{Moustakides2}, and, {more recently}, in \cite{Chronopoulou}, it was finally established for arbitrary processes with continuous paths. In both cases the optimality was established under a convenient modification of Lorden's criterion, based on the Kullback-Leibler divergence, that coincides with Lorden's criterion when the quadratic variation of the process is proportional to time. 


For continuous-time processes with jumps, the current body of work is much more limited. In fact, to our knowledge the only available optimality result is for a proportional change in the intensity of doubly stochastic Poisson processes \cite{ElKaroui}, with Lorden's expected delay criterion replaced by the expected number of jumps until detection, {motivated by applications in actuarial science.
This result includes the important case of a change in the jump intensity of a homogeneous Poisson process, for which the delay measure proposed in \cite{ElKaroui} coincides with Lorden's criterion.} We also mention a recent nonparametric result {for jump processes} \cite{German}, based on the empirical tail integral of the jump-measure, and a separate stream of literature concerning change-point detection for Poisson processes in the Bayesian setting described above (see \cite{Bayraktar}, \cite{Peskir}, \cite{Plummer}, and references therein). 

The proofs of the aforementioned results do not appear to extend in an obvious way to more general jump processes. For instance, a fundamental step in the methodology of \cite{Moustakides2} for continuous processes, as well as in \cite{ElKaroui} for doubly stochastic Poisson processes, is to use stochastic calculus to characterize the CUSUM performance functions (i.e., the average-run-length as described in Remark \ref{logThc}-(i) below) in terms of the solutions of certain differential equations, or delayed differential equations (DDE). In particular, the proof in \cite{ElKaroui} uses scale functions from the theory of L\'evy processes to deal with the aforementioned DDEs, and resolves a long-standing discontinuity problem in the methodology of Moustakides (cf.\ \cite[Sec.\ 6.4.4]{Olympia}) using the concept of a discontinuous local time, both of which may prove difficult to extend to more general jump processes (see \cite{ElKaroui} for a further discussion, and \cite{Albrecher} for another application of scale function in sequential testing).

In this work we show that CUSUM is indeed optimal for detecting a change in the statistical properties of processes with independent and {stationary increments}, i.e.\ L\'evy processes. {This result is in some sense the most natural continuous-time counterpart} of the discrete-time problem considered by Moustakides in \cite{Moustakides}. 
In addition to being of theoretical interest, it also has practical implications, as L\'evy processes form a tractable and flexible family of stochastic models with jumps, that is {well suited} to model random phenomena that exhibit erratic and discontinuous behavior. {Indeed, since the turn of the century, L\'evy processes} have found numerous applications in areas as diverse as finance and insurance, physics, and biology.  

Our approach to the problem has two main steps. 
First, we consider a continuous-time problem where the change-point is assumed to take values in a discrete set, and for which the methodology of Moustakides \cite{Moustakides} can be adapted. We show that a discretized version of the CUSUM procedure is optimal in this case, which is of practical interest in its own right, for instance in financial markets where the change-point may be assumed to occur at the beginning of a new business day. 
The second step consists in increasing the sampling frequency, and using a limiting procedure to establish the optimality of CUSUM for the continuous-time detection problem with no restriction on the value of the change-point. 
This latter part of the proof is novel and relatively 
general; it relies on little more than {standard} 
pathwise properties of L\'evy processes, and, unlike the approach in \cite{ElKaroui}, does not require separate analysis depending on whether there is a rise or a decline in the jump intensity, in addition to including changes in more general L\'evy processes. 
The trade-off is that one does not obtain as a byproduct semi-explicit expressions for the CUSUM performance functions, that are at the center of the methodology developed in \cite{ElKaroui,Moustakides2} and described above. On the other hand, we believe that our approach can be extended in various important ways, such as to incorporate exponential delay penalties (cf.\ \cite{HVPoor}), and to derive optimal stopping times for more general point processes, such as Hawkes processes. This is left for further research.


The remainder of this paper has two main sections. Section \ref{SetupCP} introduces the probabilistic framework and the notation needed to study change-point detection for L\'evy processes. Section \ref{secLorden} then reviews Lorden's change-point problem for discrete-time processes, as introduced in \cite{Lorden}, before defining the analogous continuous-time problem and presenting our  optimal change-detection results for L\'evy processes. Proofs of ancillary results are deferred to an appendix. 

\section{Probabilistic framework}\label{SetupCP}
Let $X^{0}:=(X^{0}_t)_{t\geq 0}$ and $X^{1}:=(X^{1}_t)_{t\geq 0}$ be L\'evy processes on $\bbr$, defined on the same complete filtered probability space $(\widetilde{\Omega},\widetilde{\calF},{(\widetilde\calF_t)_{t\geq 0}},\widetilde{\bbp})$, with generating triplets $(\sigma^{(0)},b^{(0)},\nu^{(0)})$ and $(\sigma^{(1)},b^{(1)},\nu^{(1)})$ relative to the truncation function ${\bf 1}_{\{|x|\leq 1\}}$ (see \cite[Sec.\ 8]{Sato}).
{In other words, $X^0$ and $X^1$ have independent and stationary increments, and trajectories that are almost surely c\`adl\`ag (right-continuous with left limits).
It is assumed that $(\sigma^{(0)},b^{(0)},\nu^{(0)})\neq(\sigma^{(1)},b^{(1)},\nu^{(1)})$, and that} we continuously observe the stochastic process $X^{(\tau)}:=(X^{(\tau)}_t)_{t\geq 0}$, defined by 
\begin{align*}
	X_t^{(\tau)}=\left\{\begin{array}{ll}  X_t^0, \quad&  t<\tau,\medskip\\
	X_t^1-X_{\tau}^1+X^0_{\tau^-}, \quad& t\geq\tau,
 \end{array}\right.
\end{align*}
where {$\tau\in\bar\bbr_0^+:=[0,\infty)\cup\{\infty\}$}, 
referred to as the change-point of the process, is assumed to be unknown and deterministic. It follows that $dX^{(\tau)}_t=dX^{0}_{t}{\bf 1}_{\{t<\tau\}}+dX^{1}_{t}{\bf 1}_{\{t\geq\tau\}}$, so the pre-change and post-change distributions of the process are determined by $X^0$ and $X^1$.
We also set $X^{(\infty)}:=X^{0}$ and $X^{(0)}:=X^{1}$, which correspond, respectively, to the cases of a change-point at time zero and no change-point. 
Finally, observe that for any $\tau\in(0,\infty)$, 
$X^{(\tau)}$ is almost surely continuous at $\tau$,
and {$(X^{(\tau)}_{t\wedge\tau})_{t\geq 0}$} and $(X^{(\tau)}_{t+\tau}-X^{(\tau)}_{\tau})_{t\geq{}0}$ are independent (stopped) L\'evy processes with the same generating triplets as $X^{0}$ and $X^1$, respectively. 

Change-point detection revolves around detecting the change-point $\tau$ as quickly and as reliably as possible, using sequential detection schemes, that is to say, a set of admissible stopping times. 
In order to formalize a framework for this problem, let us introduce 
the space of c\`adl\`ag functions $\omega:[0,\infty)\to\bbr$, denoted by $\Omega=\bbd([0,\infty),\bbr)$,
along with the canonical process $X:=(X_t)_{t\geq 0}$, defined by
\begin{align}\label{Xcan}
X_{t}(\omega):=\omega(t), \quad(\omega,t)\in\Omega\times{[0,\infty),}
\end{align}
and let $\calF_{t}$ (resp.\ $\calF$) be the smallest $\sigma$-field that makes $({X}_{s})_{s\leq{}t}$ (resp.\ $({X}_{s})_{s\geq{}0}$) measurable. As customary, let $\calF_{t^{-}}:=\sigma({\cup}_{s<t}\calF_{t})$, 
for $t>0$, and $\calF_{0^-}\equiv \calF_0$, where $\calF_0$ is the trivial $\sigma$-algebra. 
Next, for each $\tau\in\bar\bbr_{0}^+$, define the probability measure $\bbp_{\tau}$ on the space $(\Omega,\calF)$ 
as
\begin{align}\label{bbptau}
\bbp_{\tau}(A):=\widetilde{\bbp}(\tilde\omega\in\widetilde\Omega:X^{(\tau)}_{\cdot}(\tilde\omega)\in A), \quad A\in\calF,
\end{align}
and denote by $\bbe_{\tau}$ the expected value w.r.t.\ to $\bbp_{\tau}$. 
Finally, make $(\Omega,\calF,(\calF_t)_{t\geq 0},\bbp_{\tau})$ a complete filtered probability space by including $\calN_{\tau}$ in $\calF_0$, where $\calN_{\tau}$ {contains} the null sets of the measure $\bbp_{\tau}$ in $\calF$. Under assumptions (i)-(iii) below, $\calN_{\tau}$ is the same set for each $\tau\in\bar{\bbr}_0^+$.

Note that for the canonical process ${X}$, Borel sets $B_1,\dots,B_n$, and time points $t_1,\dots,t_n$, we have
\begin{align*}
\bbp_{\tau}({X}_{t_{1}}\in B_{1},\dots,  X_{t_{n}}\in B_{n})
=\widetilde{\bbp}(X^{(\tau)}_{t_{1}}\in B_{1},\dots, X^{(\tau)}_{t_{n}}\in B_{n}),
\end{align*}
so the distribution of $X$ under $\bbp_{\tau}$ is the same as the distribution of $X^{(\tau)}$ under $\widetilde\bbp$. In particular, under $\bbp_{\tau}$ with $\tau\in{(0,\infty)}$, 
the processes $({X}_{t\wedge\tau})_{t\geq 0}$ and $({X}_{t+\tau}-{X}_{\tau})_{t\geq{}0}$ are independent (stopped) L\'evy processes with {generating} triplets $(\sigma^{(0)},b^{(0)},\nu^{(0)})$ and $(\sigma^{(1)},b^{(1)},\nu^{(1)})$, respectively.
 The process $X$ can therefore be referred to as the observed process, 
with the {data-generating} probability measure unknown.

It is {also} assumed that the probability measures $\bbp_{\infty}$ and $\bbp_{0}$ induced on the path space $\Omega$ by the L\'evy processes $X^{(\infty)}$ and $X^{(0)}$, sometimes termed the \emph{in-control} and \emph{out-of-control} measures, are mutually absolutely continuous. Equivalently, it is assumed that their generating triplets satisfy the following conditions (see \cite[Thm.\ 33.1]{Sato}):
\begin{enumerate}
	\item[(i)]
	The Brownian volatilities are equal: $\sigma^{(0)}=\sigma^{(1)}$.
	\item[(ii)]
	The L\'evy measures $\nu^{(0)}$ and $\nu^{(1)}$ are equivalent and satisfy
	\begin{align}\label{levyCond} 
	\int_{\bbr_0}\big(e^{{\varphi(x)}/{2}}-1\big)^2\nu^{(0)}(dx)<\infty,
	\end{align}
	where $e^{\varphi(x)}=d\nu^{(1)}/d\nu^{(0)}$ is the Radon-Nikod\'ym derivative of $\nu^{(1)}$ w.r.t.\ $\nu^{(0)}$. 
	\item[(iii)] 
	The drift parameters $b^{(0)}$ and $b^{(1)}$ are such that 
	\begin{align}\label{driftCond}
	b^{(1)}-b^{(0)}-\int_{|x|\leq 1}x(\nu^{(1)}-\nu^{(0)})(dx)=\alpha(\sigma^{(0)})^2,
	\end{align}
	for some $\alpha\in\bbr$, and $\alpha=0$ if $\sigma^{(0)}=0$. 
\end{enumerate}
Under these conditions, each member of the family of measures $\{\bbp_{\tau},\,\tau\in\bar\bbr_0^+\}$ is absolutely continuous with respect to $\bbp_{\infty}$. It follows that for each ${\tau\geq 0}$ the likelihood ratio process
\begin{align}\label{Lst}
L_{t}^{(\tau)}:=\frac{\left.d\bbp_{\tau}\right|_{\calF_{t}}}{\left.d\bbp_{\infty}\right|_{\calF_{t}}},\quad t\geq 0,
\end{align}
is well defined, with $L_t^{(\tau)}=1$ for $t\leq{\tau}$, while for $t\geq {{\tau}}$ it can be written in terms of the likelihood ratios $L^{(0)}_{\tau}$ and $L^{(0)}_{t}$ (see the appendix for a justification): 
\begin{align}\label{LstRatio}
L_t^{(\tau)} = \Big.\frac{d\bbp_0|_{\calF_{t}}}{d\bbp_{\infty}|_{\calF_{t}}}\Big/\frac{d\bbp_0|_{\calF_{\tau}}}{d\bbp_{\infty}|_{\calF_{\tau}}}
=\frac{L_{t}^{(0)}}{L_{\tau}^{(0)}},\quad t\geq \tau.
\end{align}
Moreover, the likelihood ratio process
\begin{align}\label{LstU} 
L_{t}^{(0)} = e^{U_t},\quad t\geq 0,
\end{align}
is a $\bbp_{\infty}$-martingale, and the log-likelihood ratio $U_t$ 
takes the following form (see 
\cite[Thm.\ 33.2]{Sato}),
\begin{align}\label{Ut}
U_t 
&=\alpha X_t^c - \half\alpha^2(\sigma^{(0)})^2t-\alpha b^{(0)}t 
+ \lim_{\epsilon\downarrow 0}\Big(\sum_{0\leq s{\leq} t:\,|\Delta X_s|>\epsilon}\varphi(\Delta X_s)-t\int_{|x|>\epsilon}(e^{\varphi(x)}-1)\nu^{(0)}(dx)\Big),
\end{align}
where $(X_t^c)_{t\geq 0}$ is the continuous part of $X$ (that is, a Brownian motion with drift), 
and $\varphi$ and $\alpha$ are as in Eqs.\ (\ref{levyCond})-(\ref{driftCond}). 
We remark that $(U_t)_{t\geq 0}$ is a L\'evy process under $\bbp_{\infty}$ and $\bbp_{0}$, 
with {generating} triplets given explicitly in terms of those of $X$ under $\bbp_{\infty}$ and $\bbp_{0}$ (see \cite[Sec.\ 33]{Sato}). In particular, the L\'evy measures are given by {$\nu^{(0)}\circ\varphi^{-1}$} and {$\nu^{(1)}\circ\varphi^{-1}$},
respectively.
Furthermore, under the measures $\bbp_{\tau}$, with $\tau\in{(0,\infty)}$, 
the processes $(U_{t\wedge\tau})_{t\geq 0}$ and $(U_{t+\tau}-U_{\tau})_{t\geq 0}$ are independent {(stopped)} L\'evy processes, with the same {generating} triplets as $(U_t)_{t\geq 0}$ under $\bbp_{\infty}$ and $\bbp_0$, {respectively}.

As mentioned above, a natural class of detection strategies corresponds to the set of stopping times with respect to the filtration generated by the observed process. Hence, for each $\gamma>0$ we define 
\begin{align}\label{Tgamma}
\calT_{\gamma}:=\{T\in\calT: \bbe_{\infty}(T)\geq{}\gamma\},
\end{align} 
where $\calT$ is the set of stopping times on $\Omega$ with respect to $(\calF_{t})_{t\geq 0}$, taking values in ${\bar\bbr_0^+}$. Also, for $\Delta>0$, let $\calT(\Delta)$ and $\calT_{\gamma}(\Delta)$ denote the corresponding subsets of ${\Delta\bar\bbz_0^+}$-valued stopping times\footnote{Let $\bbz_0^+:=\{0,1,\dots\}$ and $\bar\bbz_0^+:=\bbz_0^+\cup\{\infty\}$.}. 
Since $\bbp_{\infty}$ is a probability measure under which $\tau=\infty$, i.e.\ under which there is no change-point, the purpose of the constraint $\bbe_{\infty}(T)\geq\gamma$ in (\ref{Tgamma}) is to serve as a lower bound on the mean time between false alarms (i.e.\ premature detection). 
Such a condition is needed since, as explained in the introduction, change-point detection involves a trade-off between the delay until detection {(i.e., the time while a change goes undetected)} and the frequency of false alarms. This {naturally} gives rise to an optimization problem,
and since our strategy to solve the continuous-time problem consists {in} approximating it by a sequence of discrete-time problems, the following section {sets out with a} 
discussion on Lorden's change-point problem in discrete time, {and then introduces} 
the corresponding problem for continuous-time stochastic processes.

\section{Lorden's change-point problem}\label{secLorden}

The minimax approach to change-point detection, wherein the change-point is assumed to be deterministic but unknown, was originally proposed by Lorden \cite{Lorden} in 1971. In {this} setting, detection delay is penalized {\emph{linearly}} via its worst-case expected value, and the frequency of false alarms is constrained by a lower bound on the expected time between such events.
In what follows we make this precise for discrete-time processes, and recall the seminal result of Moustakides \cite{Moustakides}, before moving on to the continuous-time case and presenting our optimal change-detection result for L\'evy processes. 

\subsection{Discrete time}\label{secLorden1}
To define Lorden's change-point problem for discrete-time stochastic processes, we need the following notation:
\begin{enumerate}
	\item[(i)] On the sample space 
	$\hat\Omega:={\bbr^{\bbn}}$, consider the canonical process $\hat{X}_{\k}(\hat\omega):=\hat\omega(\k)$, for $\hat{\omega}\in\hat{\Omega}$ and $k\geq 1$, and the natural filtration $(\hat\calF_{\k})_{k\geq 0}$ defined by $\hat\calF_0:=\{\hat{\Omega},\emptyset\}$, $\hat\calF_{k}:=\sigma(\hat {X}_{1},\dots,\hat { X}_{k})$, for $k\geq 1$, and $\hat\calF_{\infty}:=\sigma(\hat { X}_{k}:k\geq{}1)$.
	\item[(ii)] For equivalent probability distributions $Q_{0}$ and $Q_{1}$ on $\bbr$, let $(\hat\bbp_{\k})_{k\geq 1}$ be a family of probability measures on $\hat\Omega$ such that, under $\hat\bbp_{\k}$,  $(\hat{X}_{i})_{i\geq 1}$ are independent with $\hat{ X}_{1},\dots,\hat{ X}_{k-1}$ having distribution $Q_{0}$ and $\hat{ X}_{\k},\hat{ X}_{\k+1},\dots$ having distribution $Q_{1}$. Let $\hat\bbp_{\infty}$ be a probability measure under which  $(\hat X_i)_{i\geq 1}$ is i.i.d.\ with distribution $Q_0$, and denote by $\hat\bbe_{k}$ (resp.\ $\hat\bbe_{\infty}$) the expected value w.r.t.\ $\hat\bbp_{k}$ (resp.\ $\hat\bbp_{\infty}$). 
	\item[(iii)] Let $\hat\calT$ be the set of {$\bar{\bbz}_0^+$-valued} stopping times $\hat\T$ on $\hat\Omega$ with respect to the filtration $(\hat\calF_{k})_{k\geq 0}$, and, for $\gamma>0$, let $\hat\calT_{\gamma}:=\{\hat\T\in\hat\calT:\hat\bbe_{\infty}(\hat\T)\geq\gamma\}$ be the subset of stopping times satisfying a lower bound on the mean time between false alarms.
\end{enumerate}
In this setting, 
$\hat\bbp_k$ is a probability measure under which {the change-point $\hat \tau$ is equal to $k$,} 
that is, under which $k$ is the first instant {that} the sequence is governed by the {post-change} distribution $Q_1$. In particular, $\hat\bbp_1$ is a measure under which the sequence is i.i.d.\ with distribution $Q_1$ (i.e., $\hat\tau=1$) and $\hat\bbp_{\infty}$ is a measure under which the sequence is i.i.d.\ with distribution $Q_0$ (i.e., $\hat\tau=\infty$).

As a set of detection strategies, we consider all stopping times $\hat\T\in\hat\calT$, and the performance of a given stopping time is evaluated in the sense of Lorden \cite{Lorden}, with a linear penalty on detection delay\footnote{The essential supremum of a random variable $X$, defined on a {generic} probability space $(\Omega,\calF,\bbp)$, is defined as ${\esssup}\,X:={ {\esssup}_{\omega\in\Omega}X(\omega)=\inf\{u\in\bbr: \bbp(X\geq{}u)=0\}}$, with the convention that $\inf\emptyset=\infty$.}, 
\begin{align}\label{Lorden}
\hat\d(\hat T):= \sup_{k\geq 1} {\esssup}\, \hat{\bbe}_{k}\big(\big.\big(\hat{T}-(k-1)\big)^{+}\big|\hat \calF_{k-1}\big).
\end{align}
That is, detection delay is penalized via its worst-case expected value under each of the measures $\hat\bbp_k$, where the worst case is taken over all realizations of the process up to (and including) time $k-1$. The desire to make $\hat\d(\hat T)$ small must be balanced with a constraint on the rate of false alarms, so Lorden's change-point detection problem is defined as the following optimization problem:
\begin{align}\label{DiscTimeLorden}
\hat\Pi_{\gamma}^d(Q_{0},Q_{1}):=\inf_{\hat{T}\in\hat\calT_{\gamma}} \hat\d(\hat\T),
\end{align}
where $\gamma>0$, and the infimum is taken over all stopping times $\hat\T$ that satisfy the constraint $\hat\bbe_{\infty}(\hat\T)\geq\gamma$ on the mean time between false alarms.

The solution to this optimization problem is the widely used CUSUM procedure, as stated in the following theorem, originally due to Moustakides \cite{Moustakides}. His methodology is based on reframing the problem so that it can be solved using the techniques of Markovian optimal stopping theory. The key step is to establish a convenient lower bound on the detection delay of a generic stopping time, and then {proving} that the lower bound is attained by CUSUM stopping times. 

\begin{thm}\label{MoustakidesThm}
	\rm{[Moustakides, 1986]} 
	Let $h\geq 0$ 
	and define the CUSUM stopping time by 
	\begin{align}\label{cs1}
	\hat\T_h^c:=\inf\{k\geq 0:\hat\S_k\geq h\},
	\end{align}
	where 
	$\hat\S_0=0$ and 
	\begin{align}\label{cs2}
	\hat\S_k  := \max_{1\leq j\leq k}\prod_{i=j}^{k}\hat\L(\hat { X}_i)=\max(\hat\S_{k-1},1)\hat\L(\hat { X}_k),\quad k\geq 1,
	\end{align}
	where $\hat L:={dQ_1}/{dQ_0}$ is the Radon-Nikod\'ym derivative of $Q_1$ with respect to $Q_0$. Then
	$\hat\T_h^c$ solves the optimization problem (\ref{Lorden})-(\ref{DiscTimeLorden}), with $\gamma=\hat\bbe_{\infty}(\hat\T_h^c)$.
\end{thm}

\begin{rem}\label{Remark1cp}
	\hfill
	\begin{itemize}
		\item[(i)]
		Note that $h>0$ implies $\gamma=\hat{\bbe}_{\infty}(\hat T_h^c)\geq 1$, since $\hat\S_0=0$, so at least one sample is needed for the barrier $h$ to be breached. Hence, the theorem {can equivalently be formulated} for a fixed rate of false alarms $\gamma\geq 1$, assuming the existence of a barrier $h>0$ such that $\hat\bbe_{\infty}(\hat\T_h^c)=\gamma$. For $0<\gamma< 1$, the optimal rule is to stop at $k=0$ w.p.\ $1-\gamma$, or stop at $k=1$ w.p.\ $\gamma$. 
		This stopping time outperforms any CUSUM rule, even after randomizing with $k=0$. That is, if $\hat\T_h^{c,p}=\hat\T_h^c$ w.p.\ $p$, and $\hat\T_h^{c,p}=0$ w.p.\ $1-p$, for some $h>0$ and $0<p<1$ such that $\bbe_{\infty}(\hat\T_h^{c,p})=\gamma$, 
		then $\hat\d(\hat\T_{h}^{c,p})=\hat\d(\hat\T_h^c)\geq 1>\gamma$. {To be precise, these stopping times based on a randomization do not belong to the set of admissible stopping times $\hat{\calT}_{\gamma}$, but that can simply be resolved by extending the probability space (see \cite[Ch.\ 5]{Kallenberg}) to include a random variable {$\hat X_0\in \hat \calF_0$} that is uniformly distributed on $[0,1]$, and that is independent of $(\hat X_k)_{k\geq 1}$ under each of the measures $\hat{\bbp}_k$.}
		\item[(ii)]
		The optimality of CUSUM hinges on the linear delay penalty in (\ref{Lorden}). This type of penalty is suitable for many applications, such as the monitoring of manufacturing processes, where the cost of discarded items grows linearly. 
		However, in other applications, it may be of interest to use a nonlinear cost function, such as in finance where the cost of an undetected change may increase exponentially. In this case, the CUSUM test can be arbitrarily unfavorable relative to the optimal test, if the rate at which delay penalty accumulates is too high relative to the rate at which information to discriminate between the pre-change and post-change distributions accumulates. However, in \cite{HVPoor} it is shown that a simple and intuitive adaptation of the CUSUM procedure is optimal when (\ref{Lorden}) is replaced by an exponential cost of delay function.  
	\end{itemize}
\end{rem}

An important implication of Theorem \ref{MoustakidesThm} is that CUSUM is optimal in Lorden's sense when sequentially observing evenly spaced increments of a continuous-time stochastic process like $X$, defined in (\ref{Xcan}), {which, under each of the measures $\bbp_{\tau}$, defined in (\ref{bbptau})}, has independent and {stationary} increments before and after the change-point $\tau$. 
To formalize this idea, 
we need to add to the notation introduced in Section \ref{SetupCP}: 
\begin{itemize}
	\item[(i)]
	For $\Delta>0$, denote by $Q_{0}^{(\Delta)}$ and $Q_{1}^{(\Delta)}$ the distributions of $X_{\Delta}$ under $\bbp_{\infty}$ and $\bbp_0$, respectively.
	\item [(ii)]
	Define the filtration $(\breve\calF_{k\Delta})_{k\geq 0}$ generated by the $\Delta$-increments of the process $X$: $\breve{\calF}_{0}:=\{\Omega,\emptyset\}$, $\breve\calF_{k\Delta}:=\sigma(\Delta_{i}{ X}: 1\leq i\leq{}k)$ for $k\geq 1$, and $\breve\calF_{\infty}:=\sigma(\Delta_{k}{ X}: k\geq 1)$, where $\Delta_i{ X}:={ X}_{i\Delta}-{ X}_{(i-1)\Delta}$, for $1\leq i\leq k$. 
	\item [(iii)] Let $\breve\calT(\Delta)$ be the set of $\Delta\bar\bbz_0^+$-valued stopping times $\breve{T}$ on $\Omega$ with respect to $(\breve\calF_{k\Delta})_{k\geq 0}$, and,
	as before, let $\breve\calT_{\gamma}(\Delta)$ be the subset of those stopping times
	that satisfy the false alarm constraint $\bbe_{\infty}(\breve{T})\geq{}\gamma$.
\end{itemize}
Note that 
under the measure $\bbp_{k\Delta}$, with $k\geq 0$, 
the sequence of increments $(\Delta_i {X})_{i\geq 1}$ consists of independent random variables whose marginal distribution changes from $Q_{0}^{(\Delta)}$ to $Q_{1}^{(\Delta)}$ after the $k$-th {increment}. That is, under $\bbp_{k\Delta}$, the random variables $\Delta_1{X},\dots,\Delta_{k}{ X}$ have distribution $Q_{0}^{(\Delta)}$, while the random variables $\Delta_{k+1}{X},\Delta_{ k+2}{X},\dots$ have distribution $Q_{1}^{(\Delta)}$. Similarly, under $\bbp_{\infty}$ the sequence $(\Delta_i {X})_{i\geq 1}$ is i.i.d.\ with distribution $Q_{0}^{(\Delta)}$. 

It then follows from Theorem \ref{MoustakidesThm} that the CUSUM stopping time
\begin{align*}
\breve\T_h^c(Q_{0}^{(\Delta)},Q_{1}^{(\Delta)}):=\inf\{k\Delta\geq 0:\breve\S_{k\Delta}\geq h\}=\Delta\inf\{k\geq 0:\breve\S_{k\Delta}\geq h\},
\end{align*} 
where {$h\geq 0$}, $\breve\S_0=0$, and 
\begin{align*}
\breve\S_{k\Delta}
:=\max_{1\leq j\leq k}\prod_{i=j}^{k}\frac{dQ_{1}^{(\Delta)}}{dQ_{0}^{(\Delta)}}(\Delta_i {X})
=\max(\breve\S_{(k-1)\Delta},1)\frac{dQ_{1}^{(\Delta)}}{dQ_{0}^{(\Delta)}}(\Delta_k {X}),\quad k\geq 1,
\end{align*}
solves the Lorden-type optimization problem defined by
\begin{align}\label{DiscTimeLordenDeltaPI}
\breve\Pi_{\gamma}^d(Q_{0}^{(\Delta)},Q_{1}^{(\Delta)}) := \inf_{\breve{T}\in\breve\calT_{\gamma}(\Delta)} \breve{d}(\breve{T},\Delta),
\end{align}
where
\begin{align}\label{DiscTimeLordenDeltaD}
\breve{d}(\breve{T},\Delta):=\sup_{k\geq 0}{\esssup}\, \bbe_{k\Delta}\big(\big.\big(\breve{T}-{k}\Delta\big)^{+}\big|\breve\calF_{ k\Delta}\big),
\end{align}
and $\gamma=\bbe_{\infty}(\breve\T_h^c(Q_{0}^{(\Delta)},Q_{1}^{(\Delta)}))$. {In the following section (see Prop.\ \ref{ContTimeLordenDelta} therein), we extend this result to a setting where rather than observing the {discrete} increments $(\Delta X_i)_{i\geq 1}$, 
	one observes the {entire trajectory of the process $X$,} 
	but the change-point is still assumed to take values in the discrete set $\Delta\bar{\bbz}_0^+$.}

\subsection{Continuous time}\label{secLorden2}
Now we return to the continuous-time framework, as introduced in Section \ref{SetupCP}. Recall that under the probability measure $\bbp_{\tau}$, the distribution of the observed process $X$, defined in (\ref{Xcan}), undergoes an abrupt shift at the change-point $\tau$, and $\tau\in\bar{\bbr}_0^+$ is assumed to be deterministic but unknown. The continuous-time analogue of Lorden's change-point detection problem (\ref{DiscTimeLorden}) can then be defined as the optimization problem
\begin{align}\label{ContTimeLorden}
\Pi_{\gamma}^{c}:=\inf_{T\in\calT_{\gamma}} d^c(T), 
\end{align}
where the infimum is taken over all stopping times $T$ with respect to the filtration generated by the observed process, that satisfy a lower bound on the mean time between false alarms, given by $\bbe_{\infty}(T)\geq\gamma$, and 
\begin{align}\label{ContLorden}
d^c(T):=\sup_{\tau\geq{}0}{\esssup}\, \bbe_{\tau}\big(\big.\big(T-\tau\big)^{+}\big|\calF_{\tau}\big),
\end{align}
so detection delay is penalized linearly via its worst-case expected value under each of the measures $\bbp_{\tau}$.

The following theorem is our main result and it shows that the continuous-time Lorden problem (\ref{ContTimeLorden})-(\ref{ContLorden}) is solved by the continuous-time analogue of the CUSUM stopping time. {The remarks that follow then discuss some extensions of the theorem, and provide examples for specific types of L\'evy processes.}

\begin{thm}\label{ContTimeLordenThm}
	Let {$h\geq 1$} and define the CUSUM stopping time by 
	\begin{align}\label{Thc}
	T_h^c:=\inf\{t\geq 0:S_t\geq h\},
	\end{align}
	where the CUSUM process $(S_t)_{t\geq 0}$ is defined by 
	\begin{align}\label{St}
	S_t := \sup_{0\leq \tau\leq t}L_{t}^{(\tau)},\quad t\geq 0,
	\end{align}
	where $L_{t}^{(\tau)}$ is the likelihood ratio defined in (\ref{Lst}). Then, $T_h^c$ solves Lorden's optimization problem (\ref{ContTimeLorden})-(\ref{ContLorden}) with $\gamma=\bbe_{\infty}(T_h^c)$.
\end{thm}

\begin{rem}\label{ThCY}
	\hfill
	\begin{itemize}
		\item[(i)]
	This theorem encompasses previously established results on a change in the drift of a Brownian motion (see, e.g., \cite{Moustakides2}), and a change in the jump-intensity of a homogeneous Poisson process (cf.\ \cite{ElKaroui}). 
	Moreover, in a unified framework it also includes changes in the statistical properties of more general L\'evy processes, such as compound Poisson processes, jump-diffusions, and L\'evy processes with infinite jump activity. 
	\item[(ii)] 
	In Section \ref{SetupCP} we assumed the processes $X^0$ and $X^1$ to be 
		c\`adl\`ag, but the theorem extends to any processes with independent and stationary increments that are continuous in probability, since such processes have unique c\`adl\`ag modifications that are identical in distribution to the original processes (cf.\ \cite[Sec.\ 11]{Sato}).
		\item[(iii)] 
		The extension 
		to multidimensional L\'evy processes is also straightforward. The proof goes through without any significant changes if $X^0$ and $X^1$ are L\'evy processes on $\bbr^d$ for some $d>1$, with {generating} triplets $(A^{(0)},b^{(0)},\nu^{(0)})$ and $(A^{(1)},b^{(1)},\nu^{(1)})$, where the Brownian covariance matrices satisfy $A^{(0)}=A^{(1)}$, the drift parameters are such that $b^{(1)}-b^{(0)}-\int_{|x|\leq 1}x(\nu^{(1)}-\nu^{(0)})(dx)=A^{(0)}\alpha$ for some $\alpha\in\bbr^d$, with $\alpha=0$ if $A^{(0)}=0$, and the L\'evy measures $\nu^{(0)}$ and $\nu^{(1)}$ are equivalent and satisfy the integrability condition (\ref{levyCond}).
	\item[(iv)]
		{Our strategy of proof is based on approximating the continuous-time problem by discrete-time problems, and 
		\begin{align*}
		\tilde d^c(T):=\sup_{\tau\geq{}0}{\esssup}\, \bbe_{\tau}\big(\big.\big(T-\tau\big)^{+}\big|\calF_{{\tau}^-}\big),\quad T\in\calT,
		\end{align*}
		can be viewed as a natural continuous-time limit of Lorden's criterion (\ref{Lorden}), where 
		$\hat{\calF}_{k-1}$ is the information set \emph{prior} to the change-point.
		However, it turns out that $\tilde d^c(T)$ coincides with Lorden's measure $d^c(T)$, 
		for any $T\in\calT$, due to the quasi-left-continuity of the filtration $(\calF_t)_{t\geq 0}$.} 

	\end{itemize}
\end{rem}

\begin{rem}\label{logThc}
	\hfill
	\begin{itemize}
	\item[(i)] The proof of Theorem \ref{ContTimeLordenThm} (see Eq.\ (\ref{dThc}) below) shows that the CUSUM stopping time is an {equalizer rule} in the sense that its performance does not depend on the value of the change-point $\tau$:
		\begin{align*}
		d^c(T_h^c):=\sup_{\tau\geq{}0}{\esssup}\, \bbe_{\tau}\big(\big.\big(T-\tau\big)^{+}\big|\calF_{\tau}\big) = \bbe_{0}(T_h^c).
		\end{align*}
		The quantities $\bbe_0(T_h^c)$ and $\bbe_{\infty}(T_h^c)$ are generally referred to as the {average-run-lengths} (ARL) under the \emph{out-of-control} and \emph{in-control} regimes $\bbp_0$ and $\bbp_{\infty}$, respectively, and are standard measures of the performance of the CUSUM {procedure}.    
	\item[(ii)] 
	The CUSUM process (\ref{St}) is also known as the maximum likelihood ratio process, and by using (\ref{LstRatio}) it is easy to see that $\hat\tau:=\sup\{t\leq T_h^c:S_t= 1\}$ is the maximum likelihood estimate for the change-point $\tau$, based on the observed process up to time $T_h^c$. The CUSUM procedure thus combines detection and estimation, which is one reason for its sustained popularity in practical applications.
	It can also be viewed as 
	{a sequential procedure for testing} the \emph{in-control} null hypothesis $H_0$ 
	against the \emph{out-of-control} alternative $H_1$, 
	with a change announced as soon as the maximum likelihood ratio test statistic (\ref{St}) breaches a prescribed barrier. This barrier  reflects the trade-off between a large ARL under $H_0$ and a small ARL under $H_1$, which are analogous to Type I and Type II error probabilities in conventional hypothesis testing.
	\item[(iii)] 
	Another useful representation of the CUSUM stopping time is
	\begin{align}\label{ThCnew}
	T_h^c =  \inf\{t\geq 0:\log(S_t)\geq\log(h)\} = \inf\{t\geq 0:Y_t\geq\bar\h\},
	\end{align}
	where $\bar\h:=\log(h)\geq 0$ for $h\geq 1$, and, from (\ref{LstRatio})-(\ref{LstU}), it follows that the 
	process $(Y_t)_{t\geq 0}$ has the form 
	\begin{align}\label{Yt}
	Y_t & := \sup_{s\leq t}(U_t-U_s) = U_t - \inf_{0\leq s\leq t}U_s,
	\end{align}
	where $(U_t)_{t\geq 0}$ is the log-likelihood process defined in (\ref{Ut}). 
	This shows that the CUSUM stopping time is the first hitting time to $[\bar h,\infty)$ 
	of the process $(U_t)_{t\geq 0}$ reflected at its running minimum. This is also referred to as the drawup process of $(U_t)_{t\geq 0}$, and it has, along with the corresponding drawdown process, 
 received considerable attention in the financial risk management literature {(see \cite{Landriault} and references therein).} 
	\item[(iv)] 
	The expression (\ref{Ut}) for $U_t$ can be written more concisely for specific L\'evy processes:
	\begin{enumerate}
		\item[(a)]
		Let $X$ be a standard Brownian motion 
		with a change in drift from $0$ to a nonzero $\mu\in\bbr$. Then,
		\begin{align}\label{UBM}
		U_t = \mu X_t - \half\mu^2 t, \quad t\geq 0,
		\end{align}
		so the process $(U_t)_{t\geq 0}$ is a Brownian motion with drift shifting from $-\mu^2/2<0$ to $\mu^2/2>0$ at the change-point $\tau$, which in turn drives the process $(Y_t)_{t\geq 0}$ to the barrier $\bar h$.
		\item[(b)]
		Let $X$ be a compound Poisson process with a linear drift $b\in\bbr$ and 
		a change in L\'evy measures from $\nu^{(0)}$ to $\nu^{(1)}$. 
		Then,
		\begin{align}\label{Upoi}
		U_t = \sum_{0\leq s\leq t}\varphi(\Delta X_s) -(\lambda^{(1)}-\lambda^{(0)})t, \quad t\geq 0,
		\end{align}
		where $\lambda^{(i)}=\nu^{(i)}(\bbr)$, $i=0,1$, are the pre-change and post-change jump intensities of $X$, and {$\varphi=\log({d\nu^{(1)}}/{d\nu^{(0)}})$}. Furthermore, if $d\nu^{(1)}/d\nu^{(0)}\equiv\lambda^{(1)}/\lambda^{(0)}$, i.e.\ only the overall jump intensity changes, then 
		\begin{align*}
		U_t = \log\Big(\frac{\lambda^{(1)}}{\lambda^{(0)}}\Big)N_t - (\lambda^{(1)}-\lambda^{(0)})t,\quad t\geq 0,
		\end{align*}
		where 
		$(N_t)_{t\geq 0}$ is a counting process with jump-intensity shifting from $\lambda^{(0)}$ to $\lambda^{(1)}$ at the change-point $\tau$. 
		If $\lambda^{(1)}<\lambda^{(0)}$ the process $(Y_t)_{t\geq 0}$ drifts continuously through the barrier $\bar h$, but if $\lambda^{(1)}>\lambda^{(0)}$ it crosses the barrier by jumping and may overshoot it.
		\item[(c)]
		Let $X$ be a jump-diffusion process, $X=X^c+X^{j}$ where $X^{{c}}$ is a standard Brownian motion with a drift shifting from $0$ to $\mu\neq 0$, and $X^{{j}}$ a compound Poisson process with a L\'evy measure changing from $\nu^{(0)}$ to $\nu^{(1)}$. 
			In that case,
		\begin{align*}
		U_t = \mu X_t^c - \half\mu^2 t + \sum_{0\leq s\leq t}\varphi(\Delta X_s^j) -(\lambda^{(1)}-\lambda^{(0)})t, \quad t\geq 0,
		\end{align*}
		which is simply the sum of the log-likelihood ratios in (\ref{UBM}) and (\ref{Upoi}). In other words, information to distinguish between the pre-change and post-change distributions accumulates independently from the continuous component and the jump component, which is simply a consequence of their independence. This extends to L\'evy processes with infinite jump activity for which the three components of the L\'evy-It\^o decomposition - the continuous component, the ``small-jump'' component, and the ``large-jump'' component - are all independent (see, e.g., \cite[Ch.\ 4]{Sato}).
		\item[(d)]
		Let $X$ be a pure-jump L\'evy process with infinite jump activity. Then,
			\begin{align*}
			U_t 
			&=\int_0^t\int_{\bbr\setminus\{0\}}\varphi(x)\bar N(dx,ds)+\beta t, \quad t\geq 0,
			\end{align*}
			where $\bar N$ is a compensated Poisson random measure with intensity measure $\nu^{(0)}(dx)dt$ under $\bbp_{\infty}$ and $\nu^{(1)}(dx)dt$ under $\bbp_0$, 
			and condition (\ref{levyCond}) implies that the stochastic integral $(U_t-\beta t)_{t\geq 0}$ 
			is a square-integrable zero-mean martingale. 
			The drift $\beta$ under $\bbp_{\infty}$ is
			\begin{align*}
			\beta^{(0)} &= -\int_{\bbr\setminus\{0\}}(e^{\varphi(x)}-1-\varphi(x))\nu^{(0)}(dx)<0,
			\end{align*}
			while under $\bbp_0$ it is 
			\begin{align*}
			\beta^{(1)} &= \beta^{(0)} + \int_{\bbr\setminus\{0\}}\varphi(x)(\nu^{(1)}-\nu^{(0)})(dx)
			= \int_{\bbr\setminus\{0\}}(e^{\varphi(x)}(\varphi(x)-1)+1)\nu^{(0)}(dx)>0, 
			\end{align*}
			which in turn pushes $(Y_t)_{t\geq 0}$ towards the barrier $\bar h$ when the change-point $\tau$ is passed. Note that condition (\ref{levyCond}) ensures that the integrals appearing in the drift coefficients are well defined.
	\end{enumerate}
	\end{itemize}
\end{rem}

As previously mentioned, the proof of Theorem \ref{ContTimeLordenThm} is based on considering a sequence of discrete-time problems. More precisely, the first step is to show that a ``discretized'' version of the CUSUM stopping time $T_h^c$ solves a change-point problem where the change-point is restricted to take values in the discrete set {$\Delta\bar\bbz_0^+$}, for some $\Delta>0$. This gives rise to an optimization problem similar to the one in (\ref{DiscTimeLordenDeltaPI})-(\ref{DiscTimeLordenDeltaD}), 
but rather than conditioning on $\breve\calF_{k\Delta}=\sigma(\Delta_{i}{ X}: 1\leq i\leq{}k)$, the $\sigma$-algebra generated by the $\Delta$-increments of the observed process, we condition on $\calF_{k\Delta}=\sigma(X_t,\,t\leq k\Delta)$, the $\sigma$-algebra generated by the paths of the process itself. {The following proposition formalizes this idea, which is a nontrivial and somewhat unexpected extension of the result of Moustakides \cite{Moustakides} for sequentially observed random variables.}

\begin{prop}\label{ContTimeLordenDelta}
	Let {$h\geq 0$} and 
	\begin{align}\label{ContTimeLordenCusum}
	T_h^c(\Delta) :=\Delta\inf\{k\geq 0:S_{k}({\Delta})\geq h\}, 
	\end{align}
	where $S_0(\Delta)=0$, and
	\begin{align}\label{SkDelta}
	S_k(\Delta) := \sup_{0\leq m<k}L_{k\Delta}^{(m\Delta)},\quad k\geq 1.
	\end{align}
	Then $T_h^c(\Delta)$ solves the optimization problem
	\begin{align}\label{ContLordenDelta}
	\Pi_{\gamma}^{c}(\Delta) := \inf_{\T\in\calT_{\gamma}(\Delta)}d(T,\Delta),
	\end{align}
	where
	\begin{align}\label{ContLordenDelta2}
	d(T,\Delta):=\sup_{k\geq{}{0}}{\esssup}\,\bbe_{k\Delta}\big(\big(\T-{k}\Delta\big)^{+}|\calF_{k\Delta}\big),
	\end{align}
	and $\gamma=\bbe_{\infty}(T_h^c(\Delta))$.
\end{prop}
\medskip
\begin{rem}\label{Remark2}
	\hfill
	\begin{itemize}
		\item[(i)] 	This proposition serves as a stepping stone in the proof of Theorem \ref{ContTimeLordenThm}, but it is also of importance in its own right. 
		It states that the CUSUM stopping time (\ref{ContTimeLordenCusum}) is optimal when continuously monitoring a process whose distribution undergoes a change at an unknown time $\tau$ that is assumed to belong 
		to a discrete set of times, and the change is also declared at one of those times. For example, in financial applications the change may reasonably be assumed to take place at the beginning of a new business day, and in quality control a similar thing can be said about the change from the \emph{in-control} state to the \emph{out-of-control} state.
		\item[(ii)]
		Remark \ref{Remark1cp}-(i) following Theorem \ref{MoustakidesThm} also applies here. That is, $\gamma=\bbe_{\infty}(T_h^c(\Delta))\geq\Delta$ for any $h>0$, so the theorem can equivalently be stated for a fixed $\gamma\geq\Delta$, assuming the existence of a barrier $h$ such that $\bbe_{\infty}(T_h^c(\Delta))=\gamma$. 
		On the other hand, for $0<\gamma< \Delta$ the optimal stopping rule is to randomize between $k=0$ and $k=\Delta$, with probabilities $1-\gamma$ and $\gamma$, respectively.
		\item[(iii)]
		As in the discrete-time case (see Eq.\ (\ref{cs2})), it is {easy to check that (\ref{LstRatio}) implies the following recursive formula for the CUSUM process (\ref{SkDelta})}:
		\begin{align}\label{Sk}
		S_k(\Delta) 
		=\max(S_{k-1}(\Delta),1)L_{k}(\Delta),\quad k\geq 1,
		\end{align}
		where for brevity we have defined $L_{k}(\Delta):=L_{k\Delta}^{((k-1)\Delta)}$. 
	\end{itemize}
\end{rem}

Before proving the proposition, we remark that it is sufficient to consider stopping times $T\in{\calT_{\gamma}(\Delta)}$ that satisfy the constraint $\bbe_{\infty}(T)=\gamma$ with equality. 
First, if $T$ satisfies $\bbe_{\infty}(T)=\infty$, then it can be excluded by choosing a sufficiently large integer $n$ such that $\gamma\leq \bbe_{\infty}(T\wedge n\Delta)<\infty$, and $d(T\wedge n\Delta,\Delta)\leq d(T,\Delta)$. 
Second, if $T$ satisfies $\gamma<\bbe_{\infty}(T)<\infty$, then we can consider a stopping time $T^{(p)}$ such that $T^{(p)}=T$ w.p.\ $p$, and $T^{(p)}=0$ w.p.\ $1-p$, where $p=\gamma/\bbe_{\infty}(T)$. Then $\bbe_{\infty}(T^{(p)})=\gamma$, and $d(T^{(p)},\Delta)\leq d(T,\Delta)$, so $T^{(p)}$ outperforms $T$, while satisfying the false alarm constraint.


After {the simplifying assumption described in the previous paragraph},
the proof follows similar steps as the proof of Theorem \ref{MoustakidesThm}, using the methodology developed by Moustakides in \cite{Moustakides}. It rests on the following two results, whose proofs are deferred to the appendix. 
The first one gives a convenient lower bound for the performance of a generic stopping time, which the CUSUM stopping time $T_h^c(\Delta)$ satisfies with equality, while the second one shows that $T_h^c(\Delta)$ is the solution to a key optimization problem. 
\begin{lem}\label{LowerBoundLemma}
	Let $T\in{\calT(\Delta)}$ such that $0<\bbe_{\infty}(T)<\infty$. Then,
	\begin{align}\label{LowerBound}
	d(T,\Delta) \geq \bar d(T,\Delta):= {\Delta}\frac{\bbe_{\infty}\Big(\sum_{k=0}^{{T/\Delta}-1}\max(S_k(\Delta),1)\Big)}{\bbe_{\infty}\Big(\sum_{k=0}^{{T/\Delta}-1}\big(1-S_{k}(\Delta)\big)^{+}\Big)},
	\end{align}
	with equality if $T=T_h^c(\Delta)$ for some $h>0$.
\end{lem}

\medskip

\begin{prop}\label{OptProbProp}
	Let $0<h<\infty$, $\gamma=\bbe_{\infty}(T_h^c(\Delta))$, and $g:[0,\infty)\rightarrow \bbr$ be a non-increasing and continuous function. Then $T_h^c$ satisfies
	\begin{align}\label{OptProb}
	\sup_{T}\bbe_{\infty}\Big(\sum_{k=0}^{T/{\Delta}-1}g(S_k(\Delta))\Big)
	=\bbe_{\infty}\Big(\sum_{k=0}^{T_h^c(\Delta)/{\Delta}-1}g(S_k(\Delta))\Big),
	\end{align} 
	where the supremum is taken over all stopping times $T\in{\calT(\Delta)}$ that satisfy $\bbe_{\infty}(T)=\gamma$.
\end{prop}

\medskip
\noindent\textbf{Proof of Proposition \ref{ContTimeLordenDelta}} 

The result is obvious for $h=0$. For $h>0$, take $g(x)=-\max(x,1)$ and $g(x)=(1-x)^+$ in (\ref{OptProb}), to see that $T_h^c(\Delta)$ simultaneously minimizes the numerator and maximizes the denominator of (\ref{LowerBound}), over all stopping times $T\in{\calT(\Delta)}$ with $\bbe_{\infty}(T)=\gamma$. From this it follows that for any such stopping time, 
\begin{align*}
d(T,\Delta)\geq \bar\d(T,\Delta) \geq \bar d(T_h^c{(\Delta)},\Delta)=d(T_h^c{(\Delta)},\Delta),
\end{align*}
which shows that $T_h^c(\Delta)$ solves the optimization problem (\ref{ContLordenDelta})-(\ref{ContLordenDelta2}).
\hfill\qed

\medskip
Before proving Theorem \ref{ContTimeLordenThm}, we introduce two lemmas. 
The first one says that the CUSUM stopping time $T_h^c$ coincides with the first hitting time of the CUSUM process to the open set $(h,\infty)$, and that $T_h^c$ changes continuously as the barrier $h$ is increased. The second lemma states that the discretized CUSUM stopping time $T_h^c(\Delta)$ converges to $T_h^c$, as the step size $\Delta$ is reduced. Note that since $\{\bbp_{\tau},\,\tau\in\bar\bbr_0^+\}$ is a family of equivalent probability measures, almost surely in the following lemmas actually holds with respect to any of those measures. Similarly, since the L\'evy measures $\nu^{(0)}$ and $\nu^{(1)}$ are assumed to be equivalent, condition (\ref{atom}) holds for both $\nu^{(0)}$ and $\nu^{(1)}$ or neither of them. 

\begin{lem}\label{lemThc}
	Let $h>1$ and assume that under the measures $\bbp_{0}$ and $\bbp_{\infty}$, the L\'evy measure of the log-likelihood ratio process $U:=(U_t)_{t\geq 0}$, defined in (\ref{Ut}), does not have an atom at $\bar h=\log(h)$. That is,
	\begin{align}\label{atom}
	{(\nu^{(i)}\circ\varphi^{-1})}(\bar h)=\nu^{(i)}(\{x\in\bbr:\varphi(x)=\bar h\})=0,  \quad i=0,1,
	\end{align}
	{where} $\varphi=\log(d\nu^{(1)}/d\nu^{(0)})$. 
	Then the following assertions hold true almost surely, under the measures $\bbp_0$ and $\bbp_{\infty}$, 
	for the CUSUM stopping time $T_h^c$ 
	defined in (\ref{Thc}):
	\begin{enumerate}
		\item[(i)]
		$T^c_h = \tau_h := \inf\{t\geq 0: S_t> h\}$.
		\item[(ii)] 
		$T^c_{h-\epsilon}\leq T^c_h$, for any $\epsilon>0$. 
		\item[(iii)] 
		$T^c_{h-\epsilon}\to T^c_h$, as $\epsilon\downarrow 0$.
	\end{enumerate}
\end{lem}
\noindent\textbf{Proof:} To prove (i), recall the representation (\ref{ThCnew})-(\ref{Yt}) for $T_h^c$ in terms of the drawup process $Y:=(Y_t)_{t\geq 0}$, i.e.\ the process $U:=(U_t)_{t\geq 0}$ reflected at its running minimum, and observe that the paths of $U$ can be decomposed into independent excursions from its running minimum, potentially interlaced by time intervals where the process can be described as drifting at its minimum\footnote{Such intervals, contributing to the Lebesgue measure of the time the process spends at its minimum, are not restricted to processes with a compound Poisson jump component. 
	For instance, any spectrally positive and bounded variation L\'evy process $X$, with generating triplet $(0,b,\nu)$, can drift at its minimum 
	if $d=b-\int_{|x|\leq 1}x\nu(dx)<0$, because in that case $X_t/t\to d$ a.s.\ as $t\to 0$ (cf.\ \cite{Sato}, p.\ 323).}. Then note that $Y_{T_h^c}\geq \bar h$, and that if $U$ is not a compound Poisson process, then 
the process $Y$ can cross the barrier $\bar h$ in two different ways, {which we now proceed to describe.}



First, $Y$ is said to creep through the barrier if $Y_{\tau_h^c}=\bar h$. 
If $U$ is an infinite variation L\'evy process, then $T_h^c=\tau_h$ follows from the strong Markov property and the point 0 being regular for $(0,\infty)$, which makes $T_h^c<\tau_h$ impossible. If $U$ is a bounded variation L\'evy process, the same argument can be used because the point 0 is regular for $(0,\infty)$ when the drift {$d^{(i)}=b^{(i)}-\int_{|x|\leq 1}x\nu^{(i)}(dx)$} of the process is positive, which is also a necessary condition for a bounded variation process to creep through a barrier with a positive probability (see Thms.\ 6.5 and 7.11 in \cite{Kyprianou}).

Second, $Y$ can cross the barrier by jumping. However, since $Y=0$ during the intermediate times when $U$ is at its minimum, 
condition (\ref{atom}) ensures that $Y$ {cannot} jump straight to the barrier $\bar h$, {so in that case $T^c_h = \tau_h$}. 
Similarly, during an excursion of $U$ from its running minimum, $Y$ breaches the barrier $\bar h$ by overshooting it, so $T^c_h = \tau_h$.
This is because for a L\'evy process $X$ that is not a compound Poisson process, $\{X_{\hat\tau_x}=x,X_{\hat\tau_x-}<x\}$ is a null event, where for $x>0$, $\hat\tau_x=\inf\{t\geq 0:X_t\geq x\}$. In other words, 
$X$ {cannot} strike a given barrier from a position strictly below it. This follows from \cite[Lem.\ 5.8]{Kyprianou} when $X$ is a subordinator, while for a general L\'evy process $X$ it holds because the range of the running maximum process, 
$\bar X_t:=\sup_{0\leq s\leq t}X_t$, coincides almost surely with the range of the ascending ladder heights process of $X$, which is a subordinator and cannot jump to the level $x$ from below it (cf.\ \cite[p.\ 219]{Kyprianou}). 

Finally, we consider the case when $U$ is a compound Poisson process, which happens when $X$ is a compound Poisson process with the same pre- and post-change jump intensity, but a different jump size distribution (see Eq.\ \ref{Upoi}). In this case condition (\ref{atom}) ensures that $Y$ cannot hit the barrier $\bar h$ starting from zero, but $Y$ can potentially do so in a finite number of jumps. 
Then $\bar h$ is said to be $\Delta$-accessible (cf.\ \cite{Picard}), but the number of such points is finite or countable,
so we can find a sequence $(\epsilon_n)_{n\geq 1}$ such that $\epsilon_n\downarrow 0$ and $\bar h-\epsilon_n$ is not $\Delta$-accessible. For such points it is clear that $T_{h-\epsilon_n}^c=\tau_{h-\epsilon_n}$, and by part (iii) of this lemma we have $T_{h-\epsilon_n}^c\to T_h^c$ as $n\to\infty$, and it follows that $T_h^c=\tau_h$.

To show (ii), note that $[h,\infty)\subset[h-\epsilon,\infty)$ for any $\epsilon>0$, so $T_{h-\epsilon}^c$ is an increasing sequence of stopping times, and $T_{h-\epsilon}^c\leq T_h^c$, for all $\epsilon>0$. Thus, the limit $T=\lim_{\epsilon\to 0}T_{h-\epsilon}^c$ is a stopping time and  $T\leq T_h^c$. {Due to quasi-left-continuity} of L\'evy processes we have $Y_{T_{h-\epsilon}^c}\to Y_{T}$ almost surely, as $\epsilon\to 0$, and since $Y_{T_{h-\epsilon}^c}\in[h-\epsilon,\infty)$ it follows that $Y_T\in [h,\infty)$, and therefore $T_h^c\leq T$. We conclude that $T=T_h^c$, so $T_{h-\epsilon}^c\to T_h^c$, as $\epsilon\to 0$, which proves (iii). Note that condition (\ref{atom}) is not needed for (ii) and (iii) to be satisfied.
\hfill\qed

\medskip

\begin{lem}\label{lemConvStop}
	Let $h>1$ and $(\Delta_n)_{n\geq 1}$ be such that $\Delta_n\bbz_0^+\subset\Delta_{n+1}\bbz_0^+$ for all $n\geq 1$, and assume that condition (\ref{atom}) is satisfied. 
	Then the following assertions hold true almost surely, under the measures $\bbp_0$ and $\bbp_{\infty}$, 
	for the CUSUM stopping time $T_h^c$ defined in (\ref{Thc}), and the stopping times $(T_h^c(\Delta_n))_{n\geq 1}$ defined in (\ref{ContTimeLordenCusum}):
	\begin{enumerate}
		\item[(i)]
		$T_h^c \leq T_h^c(\Delta_{n+1}) \leq T_h^c(\Delta_n),\quad  n\geq 1$. 
		\item[(ii)]
		$T_h^c(\Delta_n)\to T_h^c,\quad n\to\infty$.
	\end{enumerate}
\end{lem}
\noindent\textbf{Proof:} Assertion (i) is clear from the definitions of $T_h^c$ and $T_h^c(\Delta_n)$. 
To show (ii), recall, as in the {proof of the} previous lemma, the representation (\ref{ThCnew})-(\ref{Yt}) for $T_h^c$ in terms of the drawup process $Y$, 
and write $T_h^c(\Delta_n)$ in a similar way as
\begin{align*}
T_h^c(\Delta_n) = \inf\{k\Delta_n\geq 0:S_k(\Delta_n)\geq h\}
=\inf\{k\Delta_n\geq 0: Y^{(\Delta_n)}_{k\Delta_n}\geq\bar\h\},
\end{align*}
where $\bar h=\log(h)$, and the discretized drawup process is defined by
\begin{align*}
Y^{(\Delta_n)}_{k\Delta_n} 
:= U_{k\Delta_n} - \inf_{0\leq m<k} U_{m\Delta_n},\quad k\geq 0.
\end{align*}
Since $U$ is a L\'evy process, its trajectories are c\`adl\`ag, and it follows that the trajectories of $Y$ {and 
	$M_t:=\inf_{s\leq t}U_s$,}
are c\`adl\`ag as well. The process $(Y_{k\Delta_n}^{(\Delta_n)})_{k\geq 0}$ can also be extended to a piecewise constant c\`adl\`ag process by {defining}
\begin{align*}
Y_t^{(\Delta_n)} := Y_{k_t^{(n)}\Delta_n}^{(\Delta_n)},\quad t\geq 0,
\end{align*}
where\footnote{For $x\in\bbr$, $\lfloor x\rfloor:=\sup\{z\in\bbz: z\leq x\}$ and $\lceil x\rceil:=\inf\{z\in\bbz: z\geq x\}$.}
$k_t^{(n)}:=\lfloor t/\Delta_n\rfloor$, and we now show that 
\begin{align}\label{Yconv}
\forall t\in\cup_{n\geq 1}\Delta_n\bbz_0^+:\quad Y_t^{(\Delta_n)}\stackrel[]{\text{a.s.}}{\longrightarrow} \,Y_t,\quad n\to\infty.
\end{align} 
Indeed, for a fixed $t_0\in\cup_{n\geq 1}\Delta_n\bbz_0^+$ we have $k_{t_0}^{(n)}\Delta_n=t_0$ for $n$ big enough,  
so $U_{k_{t_0}^{(n)}\Delta_n}=U_{t_0}$. The definition of $Y_{t_0}^{(\Delta_n)}$ then shows that a sufficient condition for the convergence $Y_{t_0}^{(\Delta_n)}\stackrel[]{\text{a.s.}}{\longrightarrow} \,Y_{t_0}$ is given by
\begin{align*}
M_{t_0}^{(\Delta_n)}:={\inf_{0\leq m<k_{t_0}^{(n)}}}U_{m\Delta_n}\stackrel[]{\text{a.s.}}{\longrightarrow} \, \inf_{0\leq s\leq {t_0}}U_s = M_{t_0}, \quad n\to \infty. 
\end{align*}
The definition of $M_{t_0}$ and the right-continuity of the process $U$ show that for any $\epsilon>0$, there exist $s_{\epsilon}\in[0,{t_0}]$ and $N_{\epsilon}\in\bbn$ such that $\s_{\epsilon}\in\Delta_n\bbz_0^+$ for all $n\geq N_{\epsilon}$,
and such that $U_{\s_{\epsilon}}<M_{t_0}+\epsilon$. It follows that $M_{t_0}^{(\Delta_n)}<M_{t_0}+\epsilon$, for all $n\geq N_{\epsilon}$, which implies that $M_{t_0}^{(\Delta_n)}\stackrel[]{\text{a.s.}}{\longrightarrow}  M_{t_0}$, as $n\to\infty$, and therefore $Y_{t_0}^{(\Delta_n)}\,\stackrel[]{\text{a.s.}}{\longrightarrow} \,Y_{t_0}$. The convergence (\ref{Yconv}) then follows from  the fact that a countable union of almost sure events is also almost sure.  

Now we show that $T_h^c(\Delta_n)\stackrel[]{\text{a.s.}}{\longrightarrow} T_h^c$, as $n\to\infty$, i.e.\ that the hitting time of $(Y_t^{(\Delta_n)})_{t\geq 0}$ to the set $[\bar h,\infty)$ converges to the corresponding hitting time of $Y$. 
By Lemma \ref{lemThc}-(i) and the right-continuity of $Y$, for any $\epsilon>0$ there exists $t_{\epsilon}\in[T_h^c,T_h^c+\epsilon)$ such that $t_{\epsilon}\in\Delta_n\bbz_0^+$ for any $n$ greater than some $N_{t_{\epsilon}}\in\bbn$, and such that $Y_{t_{\epsilon}}>\bar\h$. By (\ref{Yconv}), $Y_{t_{\epsilon}}^{(\Delta_n)}\to Y_{t_{\epsilon}}$, as $n\to\infty$, so there exists $\N_{\epsilon}\in\bbn$ such that 
$Y^{(\Delta_n)}_{t_{\epsilon}}>\bar\h$, for $n\geq \N_{\epsilon}$. Thus, $T_h^c(\Delta_n)<T_h^c+\epsilon+\Delta_n$ for any $N\geq N_{\epsilon}$, which implies that $T_h^c(\Delta_n)\to T_h^c$, as $n\to\infty$. 
\hfill\qed

\medskip
We are now ready to prove Theorem \ref{ContTimeLordenThm} and thus show that the CUSUM stopping time $T_h^c$ solves the continuous-time version of Lorden's change-point problem.
\smallskip 

\noindent\textbf{Proof of Theorem \ref{ContTimeLordenThm}}

\noindent 
Let $(\Delta_n)_{n\geq 1}$ and $(T_h^c(\Delta_n))_{n\geq 1}$ be as in Lemma \ref{lemConvStop}, and assume that condition (\ref{atom}) is satisfied. It is shown in the proof of Lemma \ref{LowerBoundLemma} (see (\ref{equalizer}) therein) that
\begin{align}\label{dThcDelta}
d(T_h^{c}(\Delta_n),\Delta_n) 
&=\sup_{k\geq 1}{\esssup}\,
\bbe_{(k-1)\Delta_n}\big(\big.\big(T_h^c(\Delta_n)-(k-1)\Delta_n\big)^{+}\big|\calF_{(k-1)\Delta_n}\big)
=\bbe_0(T_h^{c}(\Delta_n)),
\end{align}
and a similar identity can be established for $T_h^c$. To see that, first note that from (\ref{LstRatio}) it follows that $L_{t}^{(\tau)}=L_{s}^{(\tau)}\cdot L_{t}^{(s)}$, for any $\tau\leq s\leq t$,
so
\begin{align}\label{StMax}
S_t &= \max\big(S_{\tau}\,L_{t}^{(\tau)},\sup_{\tau\leq s\leq t}L_{t}^{(s)}\big).
\end{align}
Since $(U_t)_{t\geq 0}$ is adapted to the filtration $(\calF_t)_{t\geq 0}$ generated by $X$, it is clear from
(\ref{LstU}) and (\ref{Ut}) 
that $L_{t}^{(s)}=e^{U_t-U_s}$, for $s\in[\tau,t]$, is measurable with respect to the filtration generated by $({X}_s-{ X}_{\tau})_{\tau\leq s{\leq} t}$, and independent of $\calF_{\tau}$. Hence, (\ref{StMax}) shows that for fixed $({X}_s-{X}_{\tau})_{\tau\leq s{\leq}t}$, $S_t$ is a non-decreasing function of $S_{\tau}$, which implies that on $\{T_h^c\geq \tau\}$, $T_h^c$ is a non-increasing function of $S_{\tau}\in\calF_{\tau}$. Thus, since $S_{\tau}\geq 1$, 
\begin{align}\label{dThc}
d^c(T_h^c) & = \sup_{\tau>{}0}{\esssup}\,\bbe_{\tau}\big(\big.\big(T_h^c-\tau\big)^{+}\big|\calF_{\tau}\big)
 = \sup_{\tau>{}0}{\esssup}\,\bbe_{\tau}\big(\big.\big(T_h^c-\tau\big)^{+}\big|S_{\tau}= 1\big)
 = \bbe_0(T_h^c),
\end{align}
where the third equality follows from the homogeneous Markov property of $(S_t)_{t\geq 0}$.
Using (\ref{dThcDelta}) and (\ref{dThc}), as well as Lemma \ref{lemConvStop}-(i), now yields
\begin{align}\label{CusumLimit}
d^c(T_h^c) = \bbe_0\big(T_h^c\big) \leq \liminf_n \bbe_0\big(T_h^{c}(\Delta_n)\big) = \liminf_n d(T_h^{c}(\Delta_n),\Delta_n),
\end{align}
and, furthermore, by Lemma \ref{lemConvStop}-(i) and (ii), 
and the monotone convergence theorem, 
\begin{align}\label{gammaConv}
\gamma_n:=\bbe_{\infty}(T_h^c(\Delta_n))
\searrow \bbe_{\infty}(T_h^c)=\gamma,\quad n\to\infty.
\end{align}
Next, for a fixed $T\in\calT_{\gamma}$, define the stopping times
\begin{align*} 
T_n:=\Big\lceil\frac{T}{\Delta_n}\Big\rceil\Delta_n+\Big\lceil\frac{\gamma_n-\gamma}{\Delta_n} \Big\rceil\Delta_n, \quad n\geq 1,
\end{align*}
which belong to $\calT_{\gamma_n}(\Delta_n)$, 
so, by Proposition \ref{ContTimeLordenDelta},
\begin{align}\label{bla}
d(T_h^{c}(\Delta_n),\Delta_n)\leq d(T_n,\Delta_n),\quad n\geq 1.
\end{align}
Moreover, using $T_n \leq T + (1+\eta_n)\Delta_n$, where $\eta_n:=\lceil{(\gamma_n-\gamma)}/{\Delta_n} \rceil$, we have
\begin{align*}
d(T_n,\Delta_n)
& = \sup_{m\geq 0}{\rm ess sup}\,\bbe_{m\Delta_n}\big(\big.\big(T_n-m\Delta_n\big)^{+}\big|\calF_{m\Delta_n}\big)\\
& \leq \sup_{m\geq 0}{\rm ess sup}\,\bbe_{m\Delta_n}\big(\big.\big( T-m\Delta_n\big)^{+}\big.|\calF_{m\Delta_n}\big)+(1+\eta_n)\Delta_n\\
& \leq \sup_{\tau\geq 0}{\rm ess sup}\,\bbe_{\tau}\big(\big.\big( T-\tau\big)^{+}\big|\calF_{\tau}\big)+(1+\eta_n)\Delta_n\\
& = d^c(T) + (1+\eta_n)\Delta_n\\
& \to d^c(T), \quad n\to\infty,
\end{align*}
since $\eta_n\Delta_n\leq \gamma_n-\gamma+\Delta_n\to 0$, as $n\to\infty$, because of (\ref{gammaConv}). This implies that
\begin{align}\label{LimsupIneq}
\limsup_n d(T_n,\Delta_n) \leq d^c(T),
\end{align}
which, together with (\ref{CusumLimit}) and (\ref{bla}), shows that
\begin{align*}
d^c(T_h^c) \leq \liminf_nd(T_h^{c}(\Delta_n),\Delta_n) \leq \limsup_nd(T_h^{c}(\Delta_n),\Delta_n) \leq \limsup_n d(T_n,\Delta_n) \leq d^c(T).
\end{align*}
In other words, for a given $T\in\calT_{\gamma}$ we have $d^c(T_h^c)\leq d^c(T)$, 
which concludes the proof when condition (\ref{atom}) of Lemma \ref{lemConvStop} is satisfied. 

If (\ref{atom}) is not satisfied, consider a sequence $(\epsilon_n)_{n\geq 1}$ such that $\epsilon_n\downarrow 0$ as $n\to\infty$, and such that the L\'evy measures {$\nu^{(i)}\circ\varphi^{-1}$} 
do not have an atom at $\bar h-\epsilon_n$, i.e.\
$\nu^{(i)}(\varphi^{-1}(\bar h-\epsilon_n))=0$, $i=0,1$, for all $n\geq 1$. This is possible because L\'evy measures are $\sigma$-finite and therefore have at most countably many atoms. In that case we have shown that $d^c(T_{h-\epsilon_n}^c)\leq d^c(T)$, for any $T\in\calT_{\gamma_{h-\epsilon_n}}$, with $\gamma_{h-\epsilon_n}:=\bbe_{\infty}(T_{h-\epsilon_n}^c)\leq \gamma$. Since $\calT_{\gamma}\subseteq \calT_{\gamma_{h-\epsilon_n}}$, it follows that $d^c(T_{h-\epsilon_n}^c)\leq d^c(T)$ is in particular true for any $T\in\calT_{\gamma}$. To complete the proof it is therefore sufficient to show that $d^c(T_{h-\epsilon_n}^c)\to d^c(T_{h}^c)$ as $n\to\infty$, which follows from Lemma \ref{lemThc}-(iii) and the dominated convergence theorem, because $d^c(T_{h-\epsilon_n}^c) = \bbe_0(T_{h-\epsilon_n}^c)$ and $d^c(T_{h}^c)=\bbe_0(T_{h}^c)$, by (\ref{dThc}).
\hfill\qed

\appendix

\section{Additional proofs}\label{appendixCP}

\noindent\textbf{Proof of (\ref{LstRatio})}

\noindent The definition of $L_{t}^{(\tau)}$ entails that
\begin{align}
\bbp_{\tau}(B) = \bbe_{\infty}({\bf 1}_{B}L_{t}^{(\tau)}), \quad \forall \,B\in\calF_{t},
\end{align}
so to prove (\ref{LstRatio}) it is sufficient to show that
\begin{align}\label{eqn0}
\bbp_{\tau}(B) = \bbe_{\infty}\big({\bf 1}_{B}\frac{L_{t}^{(0)}}{L_{\tau}^{(0)}}\big),\quad \forall B\in\calF_{t}.
\end{align}
First assume that $B\in\calF_{t}$ is of the form
\begin{align}\label{simple}
B = \{X_{t_1}\in A_1,\dots,X_{t_k}\in A_k,\dots,X_{t_n}\in A_n\},
\end{align}
for some $n\geq 1$, $0\leq t_1<\dots<t_{k-1}\leq\tau<t_k<\dots<t_n \leq t$, and Borel sets $A_1,\dots,A_n$, and write $B=B_{\tau}\cap B_{t-\tau}$, {where 
	\begin{align}\label{Bdecomp}
	B_{\tau}=\{X_{t_1}\in A_1,\dots,X_{t_{k-1}}\in A_k\}\in\calF_{\tau},\quad B_{t-\tau}=\{X_{t_k}\in A_k,\dots,X_{t_n}\in A_n\}.
	\end{align} 
	Then,} using the definitions of $L_{t}^{(0)}$ and $L_{\tau}^{(0)}$, and $\bbe_{0}({\bf 1}_{B}|\calF_{\tau})\in\calF_{\tau}$, we have
\begin{align}\label{eqn1}
\bbe_{\infty}\big({\bf 1}_{B}\frac{L_{t}^{(0)}}{L_{\tau}^{(0)}}\big)
&= \bbe_{0}\big({\bf 1}_{B}\frac{1}{L_{\tau}^{(0)}}\big)\nonumber\\
&= \bbe_0\big.\big(\frac{1}{L_{\tau}^{(0)}}\bbe_{0}({\bf 1}_{B}|\calF_{\tau})\big)\nonumber\\
&= \bbe_{\infty}\big.\big(\bbe_{0}({\bf 1}_{B}|\calF_{\tau})\big)\nonumber\\
&=\bbe_{\tau}\big.\big(\bbe_{0}({\bf 1}_{B}|\calF_{\tau})\big),
\end{align}
where the final equality {above uses} that $\bbp_{\infty}|_{\calF_{\tau}}=\bbp_{\tau}|_{\calF_{\tau}}$.
Next, recall the definition of $B_{t-\tau}$ from (\ref{Bdecomp}) and note that
\begin{align}
\bbe_{0}({\bf 1}_{B_{t-s}}|\calF_{s})
&= \bbe_{0}({{\bf 1}_{\{X_{t_k}\in A_k,\dots,X_{t_n}\in A_n\}}}|\calF_{\tau})\nonumber\\
&= \bbe_{0}({{\bf 1}_{\{X_{t_k}-X_{\tau}\in A_k{(X_{\tau})},\dots,X_{t_n}-X_{\tau}\in A_n{(X_{\tau})}\}}}|\calF_{ s})\nonumber\\
&= \bbe_{s}({{\bf 1}_{\{X_{t_k}-X_{\tau}\in A_k{(X_{\tau})},\dots,X_{t_n}-X_{\tau}\in A_n{(X_{\tau})}\}}}|\calF_{ s})\nonumber\\
&= \bbe_{s}({{\bf 1}_{\{X_{t_k}\in A_k,\dots,X_{t_n}\in A_n\}}}|\calF_{\tau})\nonumber\\
&=\bbe_{s}({\bf 1}_{B_{t-\tau}}|\calF_{\tau}),
\end{align}
where {$A{(x)}:=A-x$}, for a Borel set A and $x\in\bbr$, and the third equality follows from the quasi-left-continuity of $X$ at $\tau$, and the fact that $X$ has independent increments, so $X_{t_i}-X_{\tau}$ has the same distribution under $\bbp_0$ and $\bbp_{\tau}$, for $i=k,\dots,n$. Continuing from (\ref{eqn1}), we thus obtain
\begin{align}\label{eqn2}
\bbe_{\tau}(\bbe_{0}({\bf 1}_{B}|\calF_{\tau}))
&=\bbe_{\tau}({\bf 1}_{B_{\tau}}\bbe_{0}({\bf 1}_{B_{t-\tau}}|\calF_{\tau}))\nonumber\\
&=\bbe_{\tau}({\bf 1}_{B_{\tau}}\bbe_{\tau}({\bf 1}_{B_{t-\tau}}|\calF_{\tau}))\nonumber\\
&= \bbp_{\tau}(B),
\end{align}
and together (\ref{eqn1}) and (\ref{eqn2}) show that (\ref{eqn0}) holds for events in $\calF_{t}$ of the form (\ref{simple}). By noting that those events form a $\pi$-system, the monotone class theorem can then be used to extend the result to any $B\in\calF_{t}$.
\hfill\qed

\medskip

\noindent\textbf{Proof of Lemma \ref{LowerBoundLemma}}

\noindent We first introduce the following notation for the performance of a given stopping time $T\in\calT(\Delta)$: \begin{align*}
d_k(T,\Delta)&:={\esssup}\,b_k(T,\Delta),\quad k\geq {0},\\
b_k(T,\Delta)&:= \bbe_{k\Delta}\big(\big(\T-{k}\Delta\big)^{+}|\calF_{k\Delta}\big),\quad k\geq {0},
\end{align*}
so $d(T,\Delta) = \sup_{k\geq{}{0}}d_k(T,\Delta)$. Then $d(T,\Delta) \geq d_k(T,\Delta) \geq b_k(T,\Delta)$, $\bbp_{\infty}$-a.s. and, thus\footnote{Here we simply use the fact that $\bbe(XY)\leq{}{\esssup}(X)\bbe(Y)$, $\bbp$-a.s., for any random variables $X$ and $Y$ defined on a generic probability space $(\Omega,\calF,\bbp)$.},
\begin{align}\label{LowerBoundIneq}
d(T,\Delta)\sum_{k={0}}^{\infty}\bbe_{\infty}\Big({\bf 1}_{\{T>k\Delta \}}\big(1-S_{k}(\Delta)\big)^{+}\Big) 
\geq
\sum_{k={0}}^{\infty}\bbe_{\infty}\Big(b_{k}(T,\Delta){\bf 1}_{\{T> k\Delta\}}\big(1-S_{k}(\Delta)\big)^{+}\Big).
\end{align}
Using $0<\bbe_{\infty}(T)<\infty$ and the monotone convergence theorem, the sum on the left-hand side can be written as
\begin{align}\label{lhs}
0<\bbe_{\infty}\Big(\sum_{k=0}^{{T/\Delta}-1}(1-S_{k}(\Delta))^{+}\Big)\leq\bbe_{\infty}({T}/{\Delta})<\infty.
\end{align}
For the right-hand side of (\ref{LowerBoundIneq}), we first write
\begin{align*}
b_k(T,\Delta) & = \bbe_{k\Delta}\big(\big(\T-{k}\Delta\big)^{+}|\calF_{k\Delta}\big)\\
& = {\Delta}\sum_{m={k+1}}^{\infty}\bbe_{k\Delta}({\bf 1}_{\{T\geq m\Delta\}}|\calF_{k\Delta})\\
& = {\Delta}\sum_{m={k+1}}^{\infty}\bbe_{\infty}\Big(\big.\prod_{l={k+1}}^{m-1}{L_l(\Delta)}{\bf 1}_{\{T\geq m\Delta\}}\big|\calF_{k\Delta}\Big)\\
& = {\Delta}\bbe_{\infty}\Big(\sum_{m={k+1}}^{{T/\Delta}}\big.\prod_{l={k+1}}^{m-1}{L_l(\Delta)}\big|\calF_{k\Delta}\Big),
\end{align*} 
with $L_l(\Delta)$ defined in (\ref{Sk}). Then, from the measurability of ${\bf 1}_{\{T> k\Delta\}}$ and $(1-S_{k}(\Delta))^{+}$ with respect to $\calF_{k\Delta}$, 
\begin{align}\label{rhs}
\sum_{k={0}}^{\infty}\bbe_{\infty}\Big(b_k(T,\Delta){\bf 1}_{\{T> k\Delta\}}(1-S_{k}(\Delta))^{+}\Big)
&= {\Delta}\bbe_{\infty}\Big(\sum_{k={0}}^{{T/\Delta}{-1}}(1-S_{ k}(\Delta))^{+}\sum_{m=k{+1}}^{{T/\Delta}}\prod_{l=k{+1}}^{m-1}L_l(\Delta)\Big)\nonumber\\
&= {\Delta}\bbe_{\infty}\Big(\sum_{m=1}^{{T/\Delta}}\sum_{k={0}}^{m-1}(1-S_{ k}(\Delta))^{+}\prod_{l=k{+1}}^{m-1}L_l(\Delta)\Big)\nonumber\\
&= {\Delta}\bbe_{\infty}\Big(\sum_{m=0}^{T/\Delta-1}\max(S_m(\Delta),1)\Big),
\end{align}
where the last step uses the identity
\begin{align} \max\left(S_m(\Delta),1\right)
=\sum_{k=0}^{m}(1-S_{k}(\Delta))^{+}\prod_{l=k+1}^{m}L_l(\Delta),
\end{align} 
which can easily be shown by induction, using the identity $\max(x,1)=x+(1-x)^+$, which holds for all $x\in\bbr$, and the recursive formula (\ref{Sk}). The inequality in (\ref{LowerBound}) now follows from (\ref{LowerBoundIneq})-(\ref{rhs}). 

To show that the inequality becomes an equality for $T_h^c(\Delta)$, note that from the recursive formula (\ref{Sk}) it follows that for $n\geq k$, and for fixed $(L_m(\Delta))_{k< m\leq n}$, $S_n(\Delta)$ is an increasing function of $\max(S_{k}(\Delta),1)$. Therefore, on the event $\{T_h^c(\Delta)\geq k\Delta\}$, $T_h^c(\Delta)$ is a non-increasing function of $\max(S_{k}(\Delta),1)$. Using that, and the homogeneous Markov property of $(S_k(\Delta))_{k\geq 1}$, we obtain
\begin{align}\label{equalizer}
d_k(T_h^c(\Delta),\Delta) & = {\esssup}\,\bbe_{k\Delta}((T_h^c(\Delta)-{k}\Delta)^{+}|\calF_{k\Delta})\nonumber\\
& = {\esssup}\,\bbe_{k\Delta}((T_h^c(\Delta)-{k}\Delta)^{+}|S_{k}(\Delta)\leq 1)\nonumber\\
& = {\esssup}\,\bbe_{0}(T_h^c(\Delta))\nonumber\\
& = d_0(T_h^c(\Delta),\Delta).
\end{align}
From this it follows that $d(T_h^c(\Delta),\Delta)=d_0(T_h^c(\Delta),\Delta)$, and moreover, 
\begin{align}\label{Ed}
\bbe_{\infty}(b_k(T_h^c(\Delta),\Delta){\bf 1}_{\{T_h^c(\Delta)\geq k\Delta\}}(1-S_{k}(\Delta))^{+}) 
= d(T_h^c(\Delta),\Delta)\bbe_{\infty}({\bf 1}_{\{T_h^c(\Delta)\geq k\Delta\}}(1-S_{k}(\Delta))^{+}),
\end{align}
since, by the same arguments as used to show (\ref{equalizer}), 
\begin{align*}
b_k(T_h^c(\Delta),\Delta) 
=\bbe_{k\Delta}\big(\big(T_h^c(\Delta)-{k}\Delta\big)^{+}|S_{k}(\Delta)\leq 1\big)
=d_0(T_h^c(\Delta),\Delta)
=d(T_h^c(\Delta),\Delta),\quad k\geq 1,
\end{align*}
on the event $\{T_h^c(\Delta)\geq k\Delta,S_{k}(\Delta)\leq 1\}$. From (\ref{Ed}) it now follows that for $T_h^c(\Delta)$ the inequality in (\ref{LowerBoundIneq}) becomes an equality.
\hfill\qed

\medskip
\noindent\textbf{Proof of Proposition \ref{OptProbProp}}

\noindent By assumption, $g(0)<\infty$, and we can also assume that $g$ is bounded from below. Otherwise, we can replace $g$ with $\bar g$ given by $\bar g(x)=\max(g(x),g(h))$, because 
\begin{align*}
\bbe_{\infty}\Big(\sum_{k=0}^{T/{\Delta}-1} g(S_k(\Delta))\Big)\leq \bbe_{\infty}\Big(\sum_{k=0}^{T/{\Delta}-1}\bar g(S_k(\Delta))\Big),
\end{align*} 
with equality when $T=T_h^c(\Delta)$. {We can also assume that $g(z_0)>g(h)$, where $z_0:=\essinf L_1(\Delta)\geq 0$, with $L_1(\Delta)$ defined in (\ref{Sk}). Otherwise, if $g(z_0)\leq g(h)$, then
	\begin{align*}
	 \bbe_{\infty}\Big(\sum_{k=0}^{T/{\Delta}-1}{\g(S_k(\Delta))}\Big)
	 \leq g(0) + g(z_0)(\bbe_{\infty}(T/\Delta) -1)
	 \leq g(0) + g(z_0)(\gamma/\Delta-1),
	\end{align*}
	 for any $T\in\calT(\Delta)$ such that $\bbe_{\infty}(T)=\gamma$, 
	with equality for $T=T_h^c(\Delta)$. }
	
	In the sequel we therefore assume that $g(z_0)>g(h)$, 
	and we reduce the problem to an unconstrained optimization problem. For $s>0$, let $(S_k^{(s)}(\Delta))_{k\geq 0}$ be defined as in {(\ref{Sk})} for $k\geq 1$, and $S_0^{(s)}(\Delta)=s$. Then, for any $\lambda\in\bbr$, define
\begin{align}\label{UnconstrainedOptProb}
V(s,\lambda):=\sup_{T\in\calT(\Delta)}\bbe_{\infty}\Big(\sum_{k=0}^{T/{\Delta}-1}{(\g(S_k^{(s)}(\Delta))-\lambda)}\Big).
\end{align}
It is sufficient to show that $T_h^c(\Delta)$ solves this problem for some $\bar\lambda\in\bbr$, because then, for $T\in\calT_{\gamma}(\Delta)$,
\begin{align*}
\bbe_{\infty}\Big(\sum_{k=0}^{T_h^c(\Delta)/{\Delta}-1}{(\g(S_k^{(s)}(\Delta))-\bar\lambda)}\Big) \geq \bbe_{\infty}\Big(\sum_{k=0}^{T/{\Delta}-1}{(\g(S_k^{(s)}(\Delta))-\bar\lambda)}\Big),
\end{align*}
{and so} 
\begin{align*} \bbe_{\infty}\Big(\sum_{k=0}^{T_h^c(\Delta)/{\Delta}-1}{\g(S_k^{(s)}(\Delta))}\Big)-\gamma\bar\lambda{/\Delta}  \geq \bbe_{\infty}\Big(\sum_{k=0}^{T/{\Delta}-1}{\g(S_k^{(s)}(\Delta))}\Big)-\gamma\bar\lambda{/\Delta},
\end{align*}
{or equivalently,} 
\begin{align*} \bbe_{\infty}\Big(\sum_{k=0}^{T_h^c(\Delta)/{\Delta}-1}{\g(S_k^{(s)}(\Delta))}\Big)  \geq \bbe_{\infty}\Big(\sum_{k=0}^{T/{\Delta}-1}{\g(S_k^{(s)}(\Delta))}\Big).
\end{align*} 
To examine the quantity $V(s,\lambda)$ in (\ref{UnconstrainedOptProb}), consider a sequence of stopping times $(\tau_r^{{(s)}})_{r\geq 0}$ such that $\tau_0^{{(s)}}=0$ and $\tau_r^{{(s)}}:=\inf\{n>\tau_{r-1}^{{(s)}}:S_n^{(s)}(\Delta)\leq 1\}$. 
Consider also the sequences $(\xi_r^{{(s)}})_{r\geq 1}$, $(\eta_r^{{(s)}})_{r\geq 1}$, and $(\omega_r^{{(s)}})_{r\geq 1}$, defined by
\begin{align*}
\xi_r^{{(s)}} := \sum_{k=\tau_{r-1}^{{(s)}}+1}^{\tau_r^{{(s)}}}g(S_k^{(s)}(\Delta)),\;\qquad
\eta_r^{{(s)}} := \tau_r^{{(s)}}-\tau_{r-1}^{{(s)}},\;\qquad
\omega_r^{{(s)}} :=\xi_r^{{(s)}}-\lambda\eta_r^{{(s)}},
\end{align*}
which, by the strong Markov property of $(S_k^{(s)}(\Delta))_{k\geq 0}$, are i.i.d.\ sequences for $0\leq s\leq 1$, while $(\xi_r^{{(s)}})_{r\geq 2}$, $(\eta_r^{{(s)}})_{r\geq 2}$, and $(\omega_r^{{(s)}})_{r\geq 2}$ are i.i.d.\ sequences for $s>1$. We will need the following two lemmas, where the second one identifies the range of values of $s$ and $\lambda$ for which $V({s},\lambda)$ is finite.
\begin{lem}\label{NuMoments}
	All $\bbp_{\infty}$-moments of the stopping time $\tau_1^{{(s)}}$ exist {for any $s\geq 0$}.
\end{lem}
\begin{proof}
	Consider first the case $s=0$. Note that $a^x\leq(1-x)+xa$ for $a\geq 0$ and $0\leq x\leq 1$, so $\alpha(x):=\bbe_{\infty}((L_{1}(\Delta))^x)\leq 1$, with $L_{1}(\Delta)$ defined in (\ref{Sk}). Since, by assumption, $L_{1}(\Delta)$ is not a constant, there exists $x_0\in(0,1)$ such that $\alpha(x_0)<1$. Then, by Markov's inequality,
	\begin{align*}
	\bbp_{\infty}(\tau_1^{{(0)}}=k)\leq \bbp_{\infty}\Big(\prod_{j=1}^{k-1}L_j(\Delta)>1\Big)=\bbp_{\infty}\Big(\prod_{j=1}^{k-1}(L_j(\Delta))^{x_0}>1\Big)\leq \alpha(x_0)^{k-1},
	\end{align*}
	and all finite moments therefore exist, because
	\begin{align*}
	\bbe_{\infty}\big((\tau_1^{{(0)}})^j\big)=\sum_{k=1}^{\infty}k^j\bbp_{\infty}(\tau_1^{{(0)}}=k)\leq \sum_{k=1}^{\infty}k^j\alpha(x_0)^{k-1}<\infty.
	\end{align*}
	Since $\tau_1^{(s)}=\tau_1^{(0)}$ for $0<s\leq 1$, the result also follows for such $s$, and the case $s>1$ is done in a similar way.
\end{proof}
\begin{lem}\label{optcond}
	Let $\lambda_0:=\bbe_{\infty}(\xi_1^{{(0)}})/\bbe_{\infty}(\tau_1^{{(0)}})$.
	\begin{itemize}
		\item[(i)] If $\lambda_0<\lambda<\infty$, 
		{then $V(s,\lambda)<\infty$ for any $s\geq 0$.}
		\item[(ii)] 
		$V({s},\lambda)\to\infty$, as $\lambda\to\lambda_0^{+}$, for any $s\geq 0$.
	\end{itemize}
\end{lem}
\begin{proof}
(i) We will show that for $\lambda_0<\lambda<\infty$ and $s\geq 0$,  
		\begin{align}\label{optcondA}
		\bbe_{\infty}\Big(\sup_{n\geq 1}\Big(\sum_{k=0}^n(g(S_k^{(s)}(\Delta))-\lambda)\Big)^+\Big)<\infty,
		\end{align}
		which is a sufficient condition for $V(s,\lambda)$ to be finite (see \cite[p.69]{Shiryayev}). Moreover, it is sufficient to show (\ref{optcondA}) for $s=0$, since $g(S_k^{(s)}(\Delta))$ is non-increasing in $s$. {To that end, first} notice that $\lambda_0$ is finite, since $|\bbe_{\infty}(\xi_1^{{(0)}})|\leq D\bbe_{\infty}(\tau_1^{{(0)}}){<\infty}$ { by Lemma \ref{NuMoments}}, where $D$ is a uniform bound on the function $g$. Next, for a fixed $n\geq 1$, find $r\geq 1$ such that $\tau_{r-1}^{{(0)}}<n\leq\tau_r^{{(0)}}$, so
	\begin{align*}
	\sum_{k=0}^{n}(g(S_k^{{(0)}}(\Delta))-\lambda)\leq \sum_{k=1}^r\omega_k^{{(0)}}+\eta_r^{{(0)}}D',
	\end{align*}
	where $D'=D+|\lambda|$. For $\lambda>\lambda_0$ we have $\bbe_{\infty}(\omega_1^{{(0)}})<0$, and let $\delta>0$ be such that $\bbe_{\infty}(\omega_1^{{(0)}})+\delta<0$. Then,
	\begin{align*}
	\Big(\sum_{k=0}^n(g(S_k^{{(0)}}(\Delta))-\lambda)\Big)^+ 
	&\leq \Big(\sum_{k=1}^r(\omega_k^{{(0)}}+\delta)\Big)^++(\eta_r^{{(0)}}D'-\delta r)^+
	\leq \Big(\sum_{k=1}^r(\omega_k^{{(0)}}+\delta)\Big)^++D'\eta_r^{{(0)}}{\bf 1}(\eta_r^{{(0)}}\geq \delta' r),
	\end{align*}
	where $\delta'=\delta/D'$. Hence,
	\begin{align*}
	\bbe_{\infty}\Big(\sup_{n\geq 1}\Big(\sum_{k=0}^n(g(S_k^{{(0)}}(\Delta))-\lambda)\Big)^+\Big)
	&\leq \bbe_{\infty}\Big(\sup_{r\geq 1}\Big(\sum_{k=1}^r(\omega_k^{{(0)}}+\delta)\Big)^+\Big)
	+D'\sum_{r=1}^{\infty}\bbe_{\infty}(\eta_r^{{(0)}}{\bf 1}(\eta_r^{{(0)}}\geq \delta' r)).
	\end{align*}
	For the first term to be bounded, sufficient conditions are given by $\bbe_{\infty}(\omega_1^{{(0)}})+\delta<0$ and 
	$\bbe_{\infty}\big(\big(\omega_1^{{(0)}}-\bbe_{\infty}(\omega_1^{{(0)}})\big)^2\big)<\infty$ (cf.\ \cite[p.92]{Siegmund}). The first one {is} satisfied, and for the second one we have
	\begin{align*}
	\bbe_{\infty}\big(\big(\omega_1^{{(0)}}-\bbe_{\infty}(\omega_1^{{(0)}})\big)^2\big)
	\leq 2\big(\bbe_{\infty}\big((\omega_1^{{(0)}})^2\big)+\big(\bbe_{\infty}\big(\omega_1^{{(0)}}\big)\big)^2\big)<\infty,
	\end{align*} 
	from the boundedness of $g$ {and} Lemma \ref{NuMoments}. To show that the second term is bounded, notice that since $(\eta_r^{{(0)}})_{r\geq 1}$ is an i.i.d.\ sequence with $\eta_1^{{(0)}}=\tau_1^{(0)}$,
	\begin{align*}
	\sum_{r=1}^{\infty}\bbe_{\infty}\big(\eta_r^{{(0)}}{\bf 1}(\eta_r^{{(0)}}\geq \delta'r)\big)\leq \frac{\bbe_{\infty}\big((\eta_1^{{(0)}})^3\big)}{(\delta')^2}\sum_{r=1}^{\infty}\frac{1}{r^2}<\infty.
	\end{align*}
	(ii) Notice that $V({s},\lambda)$ is non-increasing in $\lambda$, so $\lim_{\lambda\to\lambda_0^{+}}V({ s},\lambda)$ is well defined for all $s\geq 0$. Now let $T_R\geq 1$ be a stopping time measurable with respect to the filtration $(\sigma(\xi_1^{{(s)}},\dots,\xi_r^{{(s)}},\eta_1^{{(s)}},\dots,\eta_r^{{(s)}}))_{r\geq 1}$. {Then, for $0\leq s\leq 1$,}
	\begin{align*}
	\sum_{k=0}^{\tau_{\T_R}^{{(s)}}-1}(\g(S_k^{(s)}(\Delta))-\lambda)\geq\sum_{k=1}^{T_R}(\xi_k^{{(s)}}-\lambda\eta_k^{{(s)}}){\bf 1}_{\{T_R>1\}}-2D'
	\geq\sum_{k=2}^{T_R}(\xi_k^{{(0)}}-\lambda\eta_k^{{(0)}})-(2+\eta_1^{(0)})D',
	\end{align*}
	{since $\xi_k^{{(s)}}=\xi_k^{{(0)}}$ and $\eta_k^{{(s)}}=\eta_k^{{(0)}}$ for $0\leq s\leq 1$. For $s>1$ we similarly have
	\begin{align*}
	\sum_{k=0}^{\tau_{\T_R}^{{(s)}}-1}(\g(S_k^{(s)}(\Delta))-\lambda) 
	&\geq \sum_{k=2}^{T_R}(\xi_k^{{(s)}}-\lambda\eta_k^{{(s)}}){\bf 1}_{\{T_R>1\}}-(2+\eta_1^{{(s)}})D'
	\end{align*}
	and since $\xi_k^{{(s)}}\ed\xi_k^{{(0)}}$ and $\eta_k^{{(s)}}\ed\eta_k^{{(0)}}$ for $k\geq 2$, and by Lemma \ref{NuMoments}, we can  find a constant $K^{(s)}<\infty$ such that}
	\begin{align*}
	V({s},\lambda)\geq\bbe_{\infty}\Big(\sum_{k={2}}^{T_R}(\xi_k^{{(0)}}-\lambda\eta_k^{{(0)}}){\bf 1}_{\{T_R>1\}}\Big)-(2+{K^{(s)}})D'.
	\end{align*}
	Thus, by considering a stopping $T_R$ such that $\bbe_{\infty}(T_R)<\infty$, taking limits with respect to $\lambda$, and applying the dominated convergence theorem, we get
	\begin{align*}
	\lim_{\lambda\to\lambda_0^{+}}V({ s},\lambda)\geq\bbe_{\infty}\Big(\sum_{k=1}^{T_R}(\xi_k^{{(0)}}-\lambda_0\eta_k^{{(0)}})\Big)-(2+{K^{(s)}})D',
	\end{align*}
	where the random variables $\xi_k^{{(0)}}-\lambda_0\eta_k^{{(0)}}$ have zero mean (and are not identically zero), so the right hand side involves the stopped value of a random walk with zero mean. Hence, by \cite[p.27]{Siegmund}, it can be made arbitrarily large by a proper choice of $T_R$. 
\end{proof}

\noindent To conclude the proof, we also need the following standard optimization result (see, e.g., Theorem 4.5 in \cite{Siegmund}): 

\begin{thm}\label{opt}
	Let $(\Omega,\calF,(\calF_k)_{k\geq 0},\bbp)$ be a filtered probability space, and $(Y_k)_{k\geq 0}$ be an adapted process such that $\bbe(\sup_{k\geq 0}Y_k^+)<\infty$.  Let $\calT$ be the set of all stopping times on $\Omega$ with respect to $(\calF_k)_{k\geq 0}$, such that $\bbe(Y_T)$ exists, and let $\calT_k:=\{T\in\calT:\bbp(T\geq k)=1\}$. 
	Then the optimization problem
	\begin{align*}
	\sup_{T\in\calT}\bbe(Y_T),
	\end{align*} 
	is solved by 
	$T_0:=\inf\{k\geq 0:Y_k=\gamma_k\}$,
	where 
	$\gamma_k:= \;\stackrel[T\in\calT_k]{}{\esssup}\bbe(Y_T|\calF_k)$ 
	is the Snell envelope of $(Y_k)_{k\geq 0}$.\footnote{For a family of random variables $(X_i)_{i\in I}$, where $I$ is an index set, defined on a generic probability space $(\Omega,\calF,\bbp)$, the essential supremum $\stackrel[i\in I]{}{\esssup}X_i$ is the smallest random variable that almost surely dominates all members of the family.} 
	\hfill\qed
\end{thm}
To use this result we note that (\ref{UnconstrainedOptProb}) can be written as 
\begin{align*}
V(s,\lambda)=\sup_{T\in\calT(\Delta)}\bbe_{\infty}(Y_{T/{\Delta}}^{(s)}),
\end{align*}
where $Y_n^{(s)}:=\sum_{k=0}^{n-1}(g(S_k^{(s)}(\Delta))-\lambda)$, for $n\geq 0$. 
Lemma \ref{optcond}-(i) (Eq.\ \ref{optcondA}) and Theorem \ref{opt} then imply that for $\lambda>\lambda_0$, the problem is solved by $T_{opt}^{(s)}{(\lambda)}:=\inf\{k\geq 0:Y_k^{(s)}=\gamma_k^{(s)}\}$, where $(\gamma_k^{(s)})_{k\geq 0}$ is defined by 
\begin{align*}
\gamma_k^{(s)}:=\;\stackrel[T\in\calT_k(\Delta)]{}{\esssup}\bbe_{\infty}(Y_{T/{\Delta}}^{(s)}|\calF_k),
\end{align*}
and $\calT_k(\Delta):=\{T\in\calT(\Delta):\bbp_{\infty}(T\geq k\Delta)=1\}$. By the homogeneous Markov property of $(S_k^{(s)}(\Delta))_{k\geq 0}$, we can write
$\gamma_k^{(s)}=Y_k^{(s)}+V(S_k^{(s)}(\Delta),\lambda)$, so
\begin{align}\label{Topt}
T_{opt}^{(s)}{(\lambda)}=\inf\{k\geq 0:V(S_k^{(s)}(\Delta),\lambda)=0\}=\inf\{k\geq 0:S_k^{(s)}(\Delta)\geq h_{\lambda}\}=T_{h_{\lambda}}^c(\Delta),
\end{align}
so the CUSUM stopping time $T_{h_{\lambda}}^c(\Delta)$ is optimal, and the barrier is given by
\begin{align*}
h_{\lambda}:=\sup\{u\geq 0:V(u,\lambda)>0\}\geq 0. 
\end{align*}
It is clear that $h_{\lambda}=0$ for $\lambda\geq g(0)$, and $h_{\lambda}=\infty$ for $\lambda\leq\lambda_0$, by Lemma \ref{optcond}-(ii). For $\lambda_0<\lambda<g(0)$ we have $h_{\lambda}<\infty$, since in that case $\bbe_{\infty}(\omega_r^{(s)})=c<0$ for $r\geq 2$, so as $n\to\infty $ the strong law of large numbers shows that $Y_n^{(s)}\to-\infty$, $\bbp_{\infty}$-a.s., which implies $\bbp_{\infty}(T_{opt}^{(s)}{(\lambda)}<\infty)=1$, and thus $h_{\lambda}<\infty$. 
	
The final step of the proof is to show that there exists a $\bar\lambda\in(\lambda_0,g(0))$ such that $h_{\bar\lambda}=h$. To that end, consider the function
\begin{align*}
	b(\lambda) = g(h) - \lambda + \bbe_{\infty}\big(V(\max\{h,1\}L_1(\Delta),\lambda)\big),\quad \lambda>\lambda_0,
\end{align*}
which is continuous in $\lambda$ since $\bbe_{\infty}(Y_{T/{\Delta}}^{(s)})$ is linear in $\lambda$ for every $T\in\calT(\Delta)$, so $V(s,\lambda)$, being the supremum over $T\in\calT(\Delta)$, is convex and therefore continuous in $\lambda$, {for $\lambda>\lambda_0$}. 
Since $b(g(0))=g(h)-g(0)\leq g(h)-g(z_0)<0$ and $b(\lambda)\to\infty$ as $b\to\lambda_0^+$, by Lemma \ref{optcond}-(ii), there exists a $\bar{\lambda}\in(\lambda_0,g(0))$ such that $b(\bar\lambda)=0$.
Moreover, since the Snell envelope satisfies the equation $\gamma_k^{(s)} = \max\{Y_k^{(s)},\bbe_{\infty}(\gamma_{k+1}^{(s)}|\calF_k)\}$, $\bbp_{\infty}$-a.s.\ for $k\geq 0$ (c.f.\ \cite[Ch.\ 4]{Siegmund}), $V(s,\lambda)$ satisfies the equation (c.f.\ \cite[p.69]{Shiryayev})
\begin{align}\label{intEq}
	V(s,\lambda) = \big(g(s) - \lambda + \bbe_{\infty}\big(V(\max\{s,1\}L_1(\Delta),\lambda)\big)\big)^+, 
\end{align}
and since $g(s)$ and $V(s,\lambda)$ are non-increasing in $s$, it follows that $V(s,\bar\lambda)=0$ for $s\geq h$. That is, it is optimal to stop when $S_k^{(s)}(\Delta)\geq h$, which is what we wanted to show. 
{We remark that it is possible for $V(s,\bar{\lambda})=0$ to hold for $s<h$, so $h_{\bar\lambda}<h$. In this case the stopping times $T_{h'}^{c}(\Delta)$ for $h'\in[h_{\bar\lambda},h]$ all optimize $V(s,\bar\lambda)$. That is to say, $T_{h'}^c(\Delta)$ is optimal in the class of stopping times $T\in\calT(\Delta)$ that satisfy the constraint $\bbe_{\infty}(T_{h'}^c(\Delta))=\gamma'$, with $\gamma':=\bbe_{\infty}(T_{h'}^c(\Delta))$. This follows from the fact that if $S_k^{(s)}(\Delta)\in[h_{\bar\lambda},h]$ for some $k\geq 0$, then $g(S_k^{(s)}(\Delta)) - \bar\lambda + \bbe_{\infty}(V(\max\{S_k^{(s)}(\Delta),1\}L_1(\Delta),\bar\lambda))=0$ and $V(S_k^{(s)},\bar\lambda)=0$, so the expected gain from continuing in an optimal way is zero.}
 \hfill\qed

\bibliographystyle{plain}


\end{document}